\documentclass{article}
\usepackage[utf8]{inputenc}

\usepackage{latexsym}
\usepackage{amssymb}
\usepackage{amsmath}
\usepackage{amsthm}
\usepackage{amsfonts}

\usepackage{graphicx}
\usepackage[english]{babel}

\usepackage{fontenc}

\usepackage{enumerate}
\usepackage{textcomp}
\usepackage{geometry}
\usepackage{mathrsfs}
\usepackage{natbib}
\usepackage{fancyhdr}
\usepackage[colorlinks=true,linkcolor=blue,citecolor=blue]{hyperref}
\usepackage{natbib}
\usepackage{graphicx}
\usepackage[all]{xy}

\font\tenmath=msbm10 \font\sevenmath=msbm7 \font\fivemath=msbm5
\newfam\mathfam \textfont\mathfam=\tenmath
\scriptfont\mathfam=\sevenmath \scriptscriptfont\mathfam=\fivemath

\font\teneusb=eusb10 \font\seveneusb=eusb7 \font\fiveeusb=eusb5
\newfam\eusbfam \textfont\eusbfam=\teneusb
\scriptfont\eusbfam=\seveneusb \scriptscriptfont\eusbfam=\fiveeusb
\def\eusb{\fam\eusbfam}
\def\frak R{\eusb R}

\font\tenams=msam10 \font\sevenams=msam7 \font\fiveams=msam5
\newfam\amsfam \textfont\amsfam=\tenams
\scriptfont\amsfam=\sevenams \scriptscriptfont\amsfam=\fiveams

\newcommand{\Logi}{{\mathrm{Log}_{P_{0}^{i}}^{G^{i}}}}
\newcommand{\Log}{{\mathrm{Log}}}
\newcommand{\Exp}{{\mathrm{Exp}}}
\newcommand{\tr}{{\mathrm{tr}}}
\newcommand{\rg}{{\mathrm{rg}}}
\newcommand{\ex}{{\mathrm{ex}}}
\newcommand{\ID}{{\mathrm{ID}}}

\newcommand{\Poi}{{P_{0}^{i}}}
\newcommand{\PoI}{{\mathcal{P}_{0}^{\mathrm{I}}}}
\newcommand{\Toi}{{\T_{0}^{i}}}

\newcommand{\FI}{{\mathcal{F}^{\mathrm{I}}}}
\newcommand{\GI}{{\mathrm{G}^{\mathrm{I}}}}
\newcommand{\St}{{\mathrm{St}}}
\newcommand{\MdR}{{\mathrm{M}_{d}\left(\mathbb{R}\right)}}

\newcommand{\Prob}{{\mathrm{Pr}}}
\newcommand{\rk}{{\mathrm{rk}}}
\newcommand{\Cut}{{\mathrm{Cut}}}

\newcommand{\R}{{\mathbb{R}}}
\newcommand{\N}{{\mathbb{N}}}
\newcommand{\T}{{\mathbb{T}}}
\newcommand{\M}{{\mathbb{M}}}
\newcommand{\K}{{\mathbb{K}}}
\newcommand{\E}{{\mathbb{E}}}
\newcommand{\B}{{\mathbb{B}}}
\newcommand{\G}{{\mathbb{G}}}

\newcommand{\V}{{\mathbb{V}}}
\newcommand{\W}{{\mathbb{W}}}

\newcommand{\bS}{{\mathbb{S}}}

\newcommand{\SnA}{{\mathbb{S}_{n}^{A}}}
\newcommand{\SnAstar}{{\left( \mathbb{S}_{n}^{A} \right)^{*}}}

\newcommand{\Sni}{{\mathbb{S}_{n}^{i}}}
\newcommand{\SniStar}{{\left(\mathbb{S}_{n}^{i}\right)^{*}}}
\newcommand{\Wni}{{\mathbb{W}_{n}^{i}}}
\newcommand{\Vni}{{V_{n}^{i}}}
\newcommand{\EnOfi}{{E_{n}^{(i)}}}
\newcommand{\IdOfi}{{I_{d}^{(i)}}}
\newcommand{\EnOfii}{{E_{n}^{(i,i)}}}

\newcommand{\eni}{{e_{n}^{i}}}

\newcommand{\gni}{{g_{n}^{i}}}
\newcommand{\hni}{{h_{n}^{i}}}
\newcommand{\pni}{{p_{n}^{i}}}
\newcommand{\goi}{{g_{0}^{i}}}
\newcommand{\Pni}{{\mathcal{P}_{n}^{i}}}
\newcommand{\Deltani }{{\Delta_{n}^{i}}}
\newcommand{\Cni}{{\mathcal{C}_{n}^{i}}}
\newcommand{\Kni}{{\mathcal{K}_{n}^{i}}}
\newcommand{\Phini}{{\Phi_{n}^{i}}}
\newcommand{\Gammani}{{\Gamma_{n}^{i}}}
\newcommand{\fAni}{{\mathfrak{A}_{n}^{i}}}
\newcommand{\Mni}{{\mathcal{M}_{n}^{i}}}
\newcommand{\Rni}{{\mathcal{R}_{n}^{i}}}

\newcommand{\Ani}{{A_{n}^{i}}}
\newcommand{\Gni}{{G_{n}^{i}}}
\newcommand{\Hni}{{H_{n}^{i}}}

\newcommand{\DoiI}{{\mathcal{D}_{0}^{i}(\mathrm{I})}}

\newcommand{\pii}{{\pi^{\mathrm{I}}}}

\newcommand{\cU}{{\mathcal{U}}}
\newcommand{\cV}{{\mathcal{V}}}
\newcommand{\cW}{{\mathcal{W}}}
\newcommand{\cG}{{\mathcal{G}}}

\newcommand{\cA}{{\mathcal{A}}}
\newcommand{\cB}{{\mathcal{B}}}

\newcommand{\cD}{{\mathcal{D}}}
\newcommand{\cE}{{\mathcal{E}}}
\newcommand{\cF}{{\mathcal{F}}}
\newcommand{\cI}{{\mathcal{I}}}
\newcommand{\cK}{{\mathcal{K}}}
\newcommand{\cM}{{\mathcal{M}}}
\newcommand{\cR}{{\mathcal{R}}}
\newcommand{\cL}{{\mathcal{L}}}
\newcommand{\cH}{{\mathcal{H}}}
\newcommand{\cP}{{\mathcal{P}}}
\newcommand{\cN}{{\mathcal{N}}}
\newcommand{\cS}{{\mathcal{S}}}

\newcommand{\fA}{{\mathfrak{A}}}

\newcommand{\fD}{{\mathfrak{D}}}

\newcommand{\wpi}{{\widetilde{\pi}}}

\newcommand{\wpsi}{{\overline{\psi}}}

\newcommand{\hK}{{\widehat{K}}}
\newcommand{\hgm}{{\widehat{\gamma}}}
\newcommand{\hF}{{\widehat{F}}}
\newcommand{\hS}{{\widehat{S}}}

\newcommand{\Dg}{{\mathrm{Diag}}}
\newcommand{\I}{{\mathrm{I}}}

\newcommand{\rnm}{{n^{-1/2}}}

\newcommand{\lra}{{\longrightarrow}}
\newcommand{\lmp}{{~\longmapsto~}}
\newcommand{\ra}{{\rightarrow}}

\newcommand{\Sym}{{\mathrm{Sym}_{d}}}

\newcommand{\Symdiff}{{\mathrm{Sym}_{d}^{\neq}}}

\newcommand{\ECM}{{\widehat{\Sigma}_{n}}}

\newtheorem{remark}{Remark}
\newtheorem{proposition}{Proposition}
\newtheorem{corollary}{Corollary}
\newtheorem{lemma}{Lemma}

\newtheorem{theorem}{Theorem}
\newtheorem{definition}{Definition}

\numberwithin{equation}{section}

\begin{document}

\author{Dimbihery Rabenoro and Xavier Pennec}

\title{A geometric framework for asymptotic inference of principal subspaces in PCA}

\date{}

\maketitle


\begin{abstract}
Consider data assumed to be iid samples from a multivariate Gaussian distribution whose covariance matrix has repeated eigenvalues. In our model, these eigenvalues are unknown but their multiplicities are supposed to be known. In this paper, we develop an asymptotic method to infer the collection of all principal subspaces together, i.e. the eigenspaces of this covariance matrix. Our approach is based on the geometry of the flag manifold to which the collection of all principal subspaces and our estimators of it belong.
\end{abstract}

\section{Introduction}

This article deals with the asymptotic study of principal component analysis (PCA), in the light of methods from Riemannian geometry. Throughout this paper, we assume that the data are iid samples from a \textit{Gaussian} random vector $X$ defined on a proba\-bility space $(\Omega, \cA, \Pr)$ and valued in $\R^{d}$, $d \geq 1$, whose covariance matrix $\Sigma$ has possibly \textit{repeated} eigenvalues. In PCA, the principal subspaces (PS's) are the eigenspaces of $\Sigma$ associated respectively to eigenvalues in decreasing order. The sequence of multiplicities of the latter is called the \textit{type} of $\Sigma$. Then, a classical problem is to estimate the PS's by deriving Central Limit Theorems (CLT's) involving the eigenvectors of sample covariance matrices. This question is addressed notably in the Anderson's celebrated paper \cite{Anderson 1963}. The matrix $\Sigma$ may have repeated eigenvalues, although the set of symmetric matrices having multiple eigenvalues is negligible. In fact, it is implicit in \cite{Anderson 1963} that $\Sigma$ is obtained by equalizing beforehand nearly equal but
distinct eigenvalues of a sample covariance matrix. In \cite{Szwagier and Pennec 2024}, a principled procedure is developed to equalize such eigenvalues and thus to select the type of $\Sigma$. 

The method of \cite{Anderson 1963} yields only an estimation of each PS separately and when its dimension is equal to $1$. This result is improved in \cite{Tyler 1981} where other CLT's, issued from Perturbation Theory \cite{Kato 1995}, provide an estimation of each PS of any dimension. However, each PS remains to be addressed separately. 
In contrast, our geometric method allows us notably to estimate the \textit{collection} of all PS together, regardless of their dimensions. Thus, the main novelty of our paper is to build statistics based on the Riemannian geometry of manifolds of linear subspaces of $\R^{d}$. Namely, each PS belongs to a Grassmannian, while the collection of them lies in a \textit{flag manifold}. Here, a flag is viewed as a collection of mutually orthogonal subspaces spanning $\R^{d}$, whose sequence of dimensions is also called the \textit{type} of the flag. This incremental subspaces representation of a flag is equivalent to the more usual one as nested sequence of linear subspaces. In our PCA setting, we denote by $\lambda_{1} > ... > \lambda_{r}$ the distinct eigenvalues of $\Sigma$, of respective multiplicities $q_{1}, ..., q_{r}$. Thus, the collection of all PS's forms a flag of type $\I=(q_{i})_{1 \leq i \leq r}$, denoted by $F^{\I}(\Sigma)$. Let $\FI$ be the set of all flags of type $\I=(q_{i})_{i}$, so that $F^{\I}(\Sigma) \in \FI$. It is well-known that $\FI$ is endowed with a structure of homogeneous space for the orthogonal group $O(d)$: see \cite{Rabenoro and Pennec 2024}. The introduction of such manifolds in a PCA setting is notably inspired by \cite{Pennec 2018}, where it is proved that PCA can be rephrased as an optimization on a flag manifold. We also refer to \cite{Mankovich Camps-Valls and Birdal 2024} for the use of flag manifolds for studying PCA. 

A first idea to estimate $F^{\I}(\Sigma)$ could be to exploit the structure of Riemannian manifold on $\FI$ introduced in \cite{Rabenoro and Pennec 2024} and, following \cite{Bhattacharya and Patrangenaru 2005}, to establish a CLT in $\FI$ in normal coordinates i.e. defined from the Riemannian Logarithm on $\FI$. However, the latter is \textit{unknown} in closed form. In fact, there is no known metric on $\FI$ for which an explicit expression of the Riemannian Logarithm is available. Here, the incremental subspaces representation of a flag allows us to overcome this lack, since it provides an embedding of $\FI$ within a product of Grassmanians. Indeed, the Riemannian Logarithm is available in closed form in any Grassmannian.

In fact, our method is a wide geometric extension of that of \cite{Anderson 1963}. For $n \geq 1$, let $\ECM$ be the sample covariance matrix from an iid sample of size $n$ of $X$. Then, in \cite{Anderson 1963}, a matrix $E_{n} \in O(d)$ of the form $E_{n}=\Gamma' C_{n}$ is considered, where $\Gamma$ and $C_{n}$ are respectively matrices of eigenvectors of $\Sigma$ and $\ECM$, such that all diagonal entries of $E_{n}$ are positive (See Section $\ref{sec2}$ for precise definitions). So, $E_{n}$ measures the deviation between the PS's and the eigenvectors of $\ECM$. Then, the asymptotic distribution of $E_{n}$ is described in \cite{Anderson 1963}. Thus, the diagonal blocks of size $q_{i} \times q_{i}$ of $E_{n}$ converge in distribution to a Haar measure on the set of orthogonal $q_{i} \times q_{i}$ matrices whose diagonal entries are positive. The other blocks of $E_{n}$ converge to Gaussian distributions, which provide Central Limit Theorems (CLT's) that allow to infer the PS's.

 Instead of considering the usual coordinates of $E_{n}$, we introduce a \textit{geometric splitting} of $E_{n}$ into two parts, for which we recover respectively Gaussian distributions and Haar measures as $n \ra \infty$. The first part yields CLT's from which we deduce our estimation of $F^{\I}(\Sigma)$. The second part provides statistics that are valued, for all $n \geq 1$, in the orthogonal groups $(O(q_{i}))_{i}$, contrarily to the diagonal blocks of $E_{n}$. We derive such a geometric splitting of any orthogonal matrix from the structure of \textit{Riemannian submersion} of the map which associates to a $q$-frame the $q$-linear subspace spanned by it, for any $1 \leq q \leq d$. \\[-2mm]

We describe briefly hereafter our method to infer $F^{\I}(\Sigma)$. Thanks to the explicit expression of the Riemannian Logarithm in any Grassmannian, we establish a CLT in a Grassmannian for each PS separately. Then, contrarily to the CLT's for each PS in \cite{Tyler 1981}, some geometric properties of Grassmannians allow us to \textit{concatenate} our CLT's to obtain a pivotal statistic for $F^{\I}(\Sigma)$, whose limiting distribution is a $\chi^{2}$ one. Namely, recall that a.s., for all $n \geq 1$, the eigenvalues of $\ECM$ are all distinct. Then, one can construct a flag of type $\I$ from the eigenvectors of $\ECM$ as follows: those associated to the $q_{1}$ largest eigenvalues are grouped together, then those to the next $q_{2}$ ones, ..., and finally those to the $q_{r}$ smallest ones. Then, the linear subspaces generated by each group form a flag $F^{\I}(\ECM) \in \FI$ which may be compared to $F^{\I}(\Sigma)$. Thus, we prove a result of the form  
\begin{equation}\label{mainCV}
\frac{n}{4} \left[ \widetilde{\fD}_{\widehat{\K}_{n}} \left( F^{\I}\left(\Sigma \right), F^{\I}(\ECM ) \right) \right]^{2} \xrightarrow[n \rightarrow \infty]{d} \chi^{2}_{\mathrm{D}^{\I}},  
\end{equation}

\noindent
where $\widetilde{\fD}_{\widehat{\K}_{n}} \left( \cdot, \cdot \right)$ is a discrepancy function on $\FI$, indexed by a statistic $\widehat{\K}_{n}$ which is a function of the spectrum of   
$\ECM$, and $\mathrm{D}^{\I}=\frac{1}{2}\left( d^{2}-\sum_{i} q_{i}^{2} \right)$. The convergence in Eq.(\ref{mainCV}) is illustrated numerically on synthetic experiments. Finally, we derive, from Eq.(\ref{mainCV}), confidence regions for $F^{\I}\left(\Sigma \right)$ that are interpreted as confidence ellipsoids in $\FI$, centered at $F^{\I}(\ECM )$. 

Thus, our confidence regions are defined directly in the manifold $\FI$. In contrast, the confidence regions usually met in the literature on estimation on manifolds are defined in a chart, as in \cite{Bhattacharya and Patrangenaru 2005}. Furthermore, in \cite{Tyler 1981}, for fixed $1 \leq i \leq r$, the confidence regions for the principal subspace $P_{i}(\Sigma)$ are based on a $\chi^{2}_{q_{i}(d-q_{i})}$ distribution. Then, the sum of their degrees of freedom is $\sum_{i} q_{i}(d-q_{i}) = 2\mathrm{D}^{\I}$, so that we obtain a gain of a factor of $2$. \\[-2mm]

This article is organized as follows. In Section $\ref{sec2}$, we present Anderson's results and a brief sketch of their proofs. In Section $\ref{sec3}$, we develop the background of geometry needed for the rest of the paper and we present notably the aforementioned geometric splitting of $E_{n}$. Then, in Section $\ref{sec4}$, we state our main result in Theorem $\ref{theoCLTgrass}$, which provides the limiting distribution of the statistics issued from our splitting of $E_{n}$ and we deduce the estimation of the flag $F^{\I}\left(\Sigma \right)$. Finally, in Section $\ref{sec5}$, we summarize our results and propose some open questions raised by them. The proof of Theorem $\ref{theoCLTgrass}$ is deferred to the Appendix.

\section{Review of Anderson's results}\label{sec2}

We keep the notations of the Introduction. We fix a matrix $\Gamma \in O(d)$ of eigenvectors of $\Sigma$, associated respectively to $\lambda_{1} , ... , \lambda_{r}$. Then, $\Gamma' \Sigma \Gamma = \Delta$ where $\Delta$ is the following $d \times d$ block-diagonal matrix: $\Delta := \Dg \left(\lambda_{1}I_{q_{1}} , ... , \lambda_{r}I_{q_{r}} \right)$. 

\subsection{Preliminaries}\label{preliminariesAnderson}

For $1 \leq i \leq r-1$, set $\overline{q}_{i} := \sum\limits_{j=1}^{i} q_{j}$. The sequence $\I=(q_{i})_{1 \leq i \leq r}$ induces a partition of the indices $\left\{1, ..., d\right\}$ into $r$ blocks $(\beta_{i})_{1 \leq i \leq r}$ as follows: 
\begin{equation*}
\beta_{1} = \left\{1, ..., q_{1} \right\} ~,~ .... ~,~ \beta_{i} = \left\{\overline{q}_{i}+1, ..., \overline{q}_{i+1} \right\} ~,~ ... ~,~  
\beta_{r} = \left\{\overline{q}_{r-1}+1, ..., d \right\}.
\end{equation*}

\noindent
The sequence $(\beta_{i})_{1 \leq i \leq r}$ is called the partition of $\left\{1, ..., d\right\}$ wrt $\I$. Let $A=(a_{k\ell})_{1 \leq k \leq \ell \leq d}$ be a $d \times d$ matrix. For fixed $1 \leq i \leq j \leq r$, consider the submatrix $A^{(i,j)}:=(a_{k\ell})_{k \in \beta_{i}, \ell \in \beta_{j}}$. Let $A^{(i)}$ be the $i$-th block of columns of $A$, i.e. $A^{(i)} = \begin{pmatrix} A^{(1,i)} & ... & A^{(r,i)} \end{pmatrix}'$. Thus, setting $I_{d}^{(i)}:=\left( I_{d}\right)^{(i)}$, 
\begin{equation}\label{ithBlock}
A^{(i)}=AI_{d}^{(i)} 
\quad \textrm{and} \quad 
A^{(i,j)}=\left(I_{d}^{(i)}\right)'AI_{d}^{(j)}.
\end{equation}

\begin{definition}\label{DefEigenvectorMap}
Let $\Sym$ be the set of all $d \times d$ symmetric matrices and $\Symdiff$ the subset of $\Sym$ of matrices whose eigenvalues are all distinct. For $S \in \Symdiff$, its spectrum is denoted by $\left\{ \mu_{1}(S) > ... > \mu_{d}(S) \right\}$. Let $\psi : \Symdiff \lra O(d)$ be the map such that for any $1 \leq k \leq d$, the $k$-th column of $\psi(S)$ is the eigenvector of norm $1$ associated to $\mu_{k}(S)$ whose $k$-th entry is non-negative. Then, the map $\psi$ is measurable and is called the eigenvector map. 
\end{definition}

\subsection{Statement of Anderson's results}

Assume that $\ECM \in \Symdiff$, which holds a.s. Set 
\begin{equation*}
E_{n} := \psi \left(T_{n} \right) \quad \textrm{where} \quad T_{n}:=\Gamma' \ECM \Gamma, \quad n \geq 1. 
\end{equation*}

\noindent
$E_{n}$ is well-defined, since $T_{n} \in \Symdiff$. Then, $C_{n}:=\Gamma E_{n}$ is an orthogonal matrix of eigenvectors of $\ECM$, also associated to eigenvalues in decreasing order. Thus, $E_{n}=\Gamma' C_{n}$ is the matrix mentioned in the Introduction, which measures the deviation between eigenvectors of $\Sigma$ (the PS's) and those of $\ECM$. The asymptotic distribution of $E_{n}=\Gamma' C_{n}$ is the main result of \cite{Anderson 1963}, stated in Theorem $\ref{MainAnderson}$ below, for which we need the following definition.

\begin{definition}\label{defHaar}
The Haar measure is the unique probability distribution on the compact group $O(d)$ that is invariant under right-multiplications. It is the uniform measure on $O(d)$. The conditional Haar invariant distribution is defined in \cite{Anderson 1963} as $2^{d}$ times the restriction of the Haar measure to the space of orthogonal matrices whose diagonal entries are non-negative. 
\end{definition}

\begin{theorem}\label{MainAnderson}
$(i)$ For $1 \leq i \leq r$, $E_{n}^{(i,i)} \xrightarrow[n \rightarrow \infty]{d} E^{i,i}$, where the distribution of $E^{i,i}$ is the conditional Haar invariant distribution.

\smallskip

\noindent
$(ii)$ For $n \geq 1$, set $F_{n} := \sqrt{n}E_{n}$. Then, for all $i \neq j$, $F_{n}^{(i,j)} \xrightarrow[n \rightarrow \infty]{d} F^{i,j}$, where the entries of $F^{i,j}$ are iid rv's $\mathcal{N}(0, \sigma_{i,j}^{2})$, of standard deviation $\sigma_{i,j} := 
\frac{\sqrt{\lambda_{i}\lambda_{j}}}{|\lambda_{i}-\lambda_{j}|}$.

\smallskip

\noindent
$(iii)$ The blocks $\left\{ E^{i,i}, F^{i,j} : 1 \leq i \leq r, 1 \leq j \leq r, i \neq j \right\}$ are mutually independent. 

\end{theorem}

\begin{remark}
As pointed out in the Introduction, Theorem $\ref{MainAnderson}$ allows to estimate separately each PS's, but only when its dimension is equal to $1$: See Appendix B in \cite{Anderson 1963}. 
\end{remark}

\subsection{Method for the proof of Theorem $\ref{MainAnderson}$}\label{methodAndersonProof}

\subsubsection{Preliminary results}

The method consists of propagating the uncertainty of the estimation of $\widehat{\Sigma}_{n}$ by $\Sigma$, provided by Theorem $\ref{CLTsym}$ hereafter which is obtained in \cite{Anderson 1963}, to derive the limiting distribution of the deviation between their respective eigenvectors measured by $E_{n}=\Gamma' C_{n}$. This propagation is realized by a generalized $\delta$-method stated in  Theorem $\ref{AndersonCrit}$ below which can be found in Section 7 of \cite{Anderson 1963}. A more actual reference is section 18.11 in \cite{Van der Vaart 2000}. In fact, our version extends slightly the latter by adding Condition $(\ref{VnSnStar})$ in the statement, because of constraints imposed by our geometric method.

\begin{theorem}\label{CLTsym}
Set $T_{n}:=\Gamma' \ECM \Gamma$. Then, the following CLT in $\Sym$ holds. 
\begin{equation*}
U_{n} := \sqrt{n}\left( T_{n} - \Delta \right)
\xrightarrow[n \rightarrow \infty]{d} U, 
\end{equation*}

\noindent
where $U$ is a random matrix whose distribution is characterized as follows: $U \in \mathrm{Sym}_{d}$, the blocks $\left\{ U^{(i,j)} : 1 \leq i < j \leq r \right\}$ are mutually independent and for $i \neq j$, the entries of $U^{(i,j)}$ are iid rv's $\mathcal{N}\left(0, s_{i,j}^{2}\right)$, of standard deviation $s_{i,j} := \sqrt{\lambda_{i}\lambda_{j}}$.
\end{theorem}

\begin{theorem}\label{AndersonCrit}
Let $\bS, \T$ be metric spaces endowed with their Borel $\sigma$-algebras. Let $(\bS_{n})_{n \geq 0}$ be a sequence of measurable subsets of $\bS$ and for all $n \geq 0$, $g_{n} : \bS_{n} \lra \T$ a measurable map. Let $(V_{n})_{n \geq 1}$ be a sequence of rv's valued in $(\bS_{n})_{n \geq 1}$ which converges in distribution to a rv $V$ valued in $\bS$, with $\Pr\left( V \in \bS_{0} \right)=1$. Assume that for all $n \geq 1$, there exists a measurable  subset $\bS_{n}^{*}$ of $\bS_{n}$ such that 
\begin{equation}\label{VnSnStar}
\Prob\left( V_{n} \in \bS_{n}^{*} \right) \xrightarrow[n \rightarrow \infty]{} 1
\end{equation}

\noindent
and that, for any $v \in \bS_{0}$ and all sequence $(v_{n})_{n \geq 1}$ valued in $\left(\bS_{n}^{*} \right)_{n \geq 1}$, 
\begin{equation}\label{pseudoCont}
v_{n} \xrightarrow[n \rightarrow \infty]{} v ~\implies~ g_{n}(v_{n}) \xrightarrow[n \rightarrow \infty]{} g_{0}(v). 
\end{equation}

\noindent
Then, $g_{n}(V_{n})$ converges in distribution to $g_{0}(V)$. 
\end{theorem}

\begin{proof}
We prove our extended version in subsection $\ref{proofAndersonCrit}$ of the Appendix. 
\end{proof}

\section{Preliminaries of geometry}\label{sec3}

In subsection $\ref{manifoldsLinearSub}$, we introduce the manifolds of li\c near subspaces in which the (collection of) PS's and their estimators lie. In subsection $\ref{subGeoParam}$, we define a geometric parametrization of $E_{n}$ that is an alternative to that by its coordinates in the standard basis of $\R^{d \times d}$, used in Anderson's paper \cite{Anderson 1963} to describe (Theorem $\ref{MainAnderson}$) the limiting distribution of $E_{n}$. In contrast, our parametrization of $E_{n}$ provides notably confidence regions for the collection of PS's, for which geometric preliminaries are presented in subsection $\ref{discrepancyFlags}$.

\subsection{Grassmannians, Stiefel and flag manifolds}\label{manifoldsLinearSub}

\subsubsection{Sets of linear subspaces}

Let $1 \leq q \leq d$. The Stiefel manifold is the set $\St_{q}$ of all $q$-frames of $\R^{d}$ i.e. of all $q$-tuples of orthonormal vectors of $\R^{d}$. Thus, 
\begin{equation*}
\St_{q} = \left\{ \cU \in \R^{d \times q} : \cU' \cU=I_{q} \right\}. 
\end{equation*}

\noindent
The Grassmannian $G_{q}$ is the set of all $q$-linear subspaces of $\R^{d}$, identified with the orthogonal projectors of rank $q$ i.e. such a projector $P$ is identified with its range 
$\rg(P)$: 
\begin{equation}\label{identifyGrass}
G_{q} = \left\{ P \in \mathbb{R}^{d \times d} : P'=P, ~P^{2}=P, ~\textrm{rank}(P) = q \right\}. 
\end{equation}

\noindent
Thus, $G_{q}$ is a subset of $\Sym$. Then, $\St_{q}$ and $G_{q}$ are linked by the map $\pi_{q}$ which associates to $\cU \in \St_{q}$, the (projector onto the) linear subspace spanned by $\cU$, defined by 
\begin{equation*}
\pi_{q} : \St_{q} \lra G_{q}, \quad \pi_{q}(\cU) = \cU \cU'. 
\end{equation*}

\begin{remark}\label{fiberPiQ}
Let $P \in G_{q}$ and $\cU_{P}$ a fixed $q$-frame spanning $\rg(P)$, i.e. $\cU_{P} \in \pi_{q}^{-1}(P)$. Then, any $\cV \in \pi_{q}^{-1}(P)$ is obtained from $\cU_{P}$ by a unique  orthonormal base change in $\rg(P)$. Thus, there exists a unique $K \in O(q)$ such that $\cV = \cU_{P}K$ i.e. $K=\cU_{P}' \cV$. 
\end{remark}

\begin{definition}
$(i)$ A \textit{flag} of $\R^{d}$ is a sequence of mutually orthogonal linear subspaces spanning $\R^{d}$, whose sequence of respective dimensions is called the \textit{type} of the flag. 

\noindent
$(ii)$ A flag whose type is $(1,1, ..., 1)$ is called a \textit{full flag}.

\noindent
$(iii)$ For $\Sigma \in \Sym$, its \textit{flag of eigenspaces} is the collection of its eigenspaces associated to eigenvalues in decreasing order, whose type is the sequence of their dimensions. 

\noindent
$(iv)$ Given a sequence $\I = (q_{i})_{1 \leq i \leq r}$ with $\sum q_{i}=d$, consider the product manifold 
\begin{equation}\label{productGrass}
\GI := \prod G^{i} \quad \textrm{where} \quad G^{i}:=G_{q_{i}}, \enskip 1 \leq i \leq r.
\end{equation}

\noindent
The \textit{flag manifold} of type $\I$ is the set $\FI$ of all flags of type $\I$, identified with a subset of $G^{\I}$: 
\begin{equation}\label{flagProj}
\FI = \left\{ \mathcal{P}=(P_{i})_{i} \in G^{\I} : ~\sum P_{i} = \R^{d}, ~P_{i}P_{j}=0, ~i \neq j  \right\}.
\end{equation}
\end{definition}

\begin{remark}
This definition of a flag is more suitable for computing with flags of eigenspaces than the usual equivalent one as a nested sequence of linear subspaces. 
\end{remark}

In the sequel, for $1 \leq i \leq r$, $\St_{q_{i}}$ and $\pi_{q_{i}}$ are respectively denoted by $\St^{i}$ and $\pi^{i}$.

\begin{definition}
$(i)$ Let $\St^{\I}_{\perp}$ be the set of all collections $\overline{\cU}=(\cU_{i})_{1 \leq i \leq r}$ of frames such that for $1 \leq i \leq r$, $\cU_{i} \in \St^{i}$ and the collection $\left( \pi^{i}(\cU_{i}) \right)_{i}$ lies in $\FI$. Define the map $(\pi^{i})_{i}$ by 
\begin{equation*}
(\pi^{i})_{i} : \St^{\I}_{\perp} \lra \FI
\enskip \textrm{with} \enskip
(\pi^{i})_{i}(\overline{\cU}) := \left( \pi^{i}(\cU_{i}) \right)_{i}. 
\end{equation*}

\noindent
$(ii)$ We say that $\Gamma \in O(d)$ \textit{spans} a flag $\cP=(P_{i})_{i} \in \FI$ when for all $1 \leq i \leq r$, $\pi^{i}\left(\Gamma^{(i)}\right) = P_{i}$.

\noindent
$(iii)$ For $Q \in O(d)$, the collection $\zeta(Q) := \left( Q^{(i)} \right)_{1 \leq i \leq r}$ lies in $\St^{\I}_{\perp}$, which defines a diffeomorphism $\zeta : O(d) \lra \St^{\I}_{\perp}$. 
\end{definition}

\subsubsection{Actions of the orthogonal group}\label{defGroupActions}

The group $O(d)$ acts on vectors of $\R^{d}$ by linear isometries, which induces an action on $\St_{q}$ by left-multiplication:
\begin{equation}\label{actionSt}
(Q, \cU) \mapsto Q\cU, \quad Q \in O(d), ~\cU \in \St_{q}.
\end{equation}

\noindent
Let $P \in G_{q}$ and $\cU_{P}$ a $q$-frame spanning $\rg(P)$. Any $Q \in O(d)$ sends $\cU_{P}$ to $Q\cU_{P}$, which spans $\rg(\pi_{q}(Q\cU_{P})) = \rg(QPQ')$. The latter is independent of $\cU_{P}$. So, $O(d)$ acts on $G_{q}$ by 
\begin{equation}\label{actionGr}
(Q, P) \mapsto Q \cdot P := QPQ', \quad Q \in O(d), ~P \in G_{q}.
\end{equation}

\noindent
The action on $G_{q}$ extends to the product $\GI$: for $Q \in O(d)$ and $\cP=(P_{i})_{1 \leq i \leq r} \in \GI$, set 
\begin{equation}\label{actionO(d)onGI}
Q*\cP := (Q \cdot P_{i})_{1 \leq i \leq r} \in \GI. 
\end{equation}

\noindent
Now, the action of $Q \in O(d)$ preserves the mutual orthogonality of linear subspaces. Thus, $\cP \in \FI \implies Q*\cP \in \FI$. So, the action defined in $(\ref{actionO(d)onGI})$ induces an action of $O(d)$ on $\FI$.

\begin{definition}
For $1 \leq i \leq r$, the standard $i$-th Grassmannian is $P_{0}^{i} \in G^{i}$ defined by
\begin{equation*}
P_{0}^{i}:=\pi^{i}\left( \left\{ \epsilon_{j} : j \in \beta_{i} \right\} \right)
\quad \textrm{i.e.} \quad
P_{0}^{i} = \mathrm{Diag}[0_{q_{1}}, ..., I_{q_{i}}, ..., 0_{q_{r}}].  
\end{equation*}

\noindent
The \textit{standard flag} of type $\I$ is the flag $\mathcal{P}_{0}^{\mathrm{I}}:=(P_{0}^{i})_{i}$. The \textit{main orbital map} is the map 
\begin{equation*}
\pii : O(d) \lra \FI, \quad \pii(Q) = Q*\PoI.  
\end{equation*}

\end{definition}

\begin{remark}\label{remarkPiS}
$(i)$ It is easily obtained from $(\ref{ithBlock})$ (see \cite{Rabenoro and Pennec 2024} for details) that 
\begin{equation}\label{linkPiSproof}
(\pi^{i})_{i} \circ \zeta = \pii \quad \textrm{i.e.} \quad \left( \pi^{i}\left(Q^{(i)}\right) \right)_{i} = \pii(Q) = \left( Q\Poi Q' \right)_{i}, \enskip Q \in O(d). 
\end{equation}

\noindent
$(ii)$ Let $Q \in O(d)$ which spans a flag $\cP \in \FI$. By definition, $\cP=\left( \pi^{i}\left(Q^{(i)}\right) \right)_{i}$. So, by $(\ref{linkPiSproof})$,  
\begin{equation}\label{remarkSpanFlag}
\cP = \pii(Q) := Q * \PoI. 
\end{equation}
\end{remark}

\subsubsection{Manifolds of linear subspaces}\label{manifoldStructure}

These actions of $O(d)$ provide the manifold structures of $\St_{q}$, $G_{q}$ and $\FI$, through Proposition $\ref{generalTheo}$ hereafter (see \cite[Proposition A.2]{Bendokat Zimmermann and Absil 2024}).

\begin{proposition}\label{generalTheo}
Consider a compact Lie group $\G$ which acts smoothly on a smooth manifold $\M$. For $b_{0} \in \M$, let $\B$ be its orbit and $\K$ its isotropy group. Let $\psi : \G \lra \G / \K$ be the canonical quotient map. Define the map $\pi_{0} : \G \lra  \B$ by $Q \lmp Q \cdot b_{0}$. Then, 

\vspace{.15cm}

\noindent
$(i)$ $\B$ is an embedded submanifold of $\M$.  

\vspace{.15cm}

\noindent
$(ii)$ The map $\widetilde{\pi}_{0} : \G / \K \lra \B$ such that $\widetilde{\pi}_{0} \circ \psi = \pi_{0}$ is a diffeomorphism. 
\end{proposition}

\noindent
$O(d)$ is a compact Lie group and its actions defined in subsection $\ref{defGroupActions}$ are smooth. In \cite{Bendokat Zimmermann and Absil 2024}, Proposition $\ref{generalTheo}$ is applied to prove that $\St_{q}$ (resp. $G_{q}$) is an embedded submanifold of $\R^{d \times q}$ (resp. $\Sym$) that is diffeomorphic to $O(d)/O(d-q)$ (resp. $O(d)/(O(q) \times O(d-q))$). 

The manifold structure of $\FI$ is provided by the following Lemma, proved in \cite{Rabenoro and Pennec 2024}.

\begin{lemma}\label{FiberPii}
The main orbital map $\pii : O(d) \lra \FI$ with $\pii(Q)=Q*\PoI $ is surjective.
\end{lemma}

\noindent
By Lemma $\ref{FiberPii}$, $\FI$ is the orbit of $\PoI$ under the smooth action of $O(d)$ on $\GI$ defined in $(\ref{actionO(d)onGI})$. Now, the isotropy group of $\PoI$ under this action is the group 
\begin{equation*}
O(\I) := \left\{ \Dg(H_{1}, ..., H_{r}) : H_{i} \in O(q_{i}), 1 \leq i \leq r \right\} \simeq \prod O(q_{i}). 
\end{equation*}

\noindent
By Proposition $\ref{generalTheo}$, $\FI$ is an embedded submanifold of $\GI$ and the map $\wpi^{\I} : O(d)/O(\I) \lra \FI$ s.t. $\pii = \wpi^{\I} \circ \psi^{\I}$ is a diffeomorphism, where $\psi^{\I} : O(d) \lra O(d)/O(\I)$ is the canonical map.

\subsubsection{Background of Riemannian geometry}

We introduce Riemannian structures on these manifolds, in order to measure intrinsic distances and perform estimation on them. \\

More generally, let $\M$ be a connected smooth manifold. A \textit{metric} on $\M$ is a collection $g=(g_{x})_{x \in \M}$ of inner products on the tangent spaces $T_{x}\M$, varying smoothly wrt $x$. Then, $g$ defines the \textit{geodesic distance} $\rho_{g}$ on $\M$, where for $x, y \in \M$, $\rho_{g}(x, y)$ is the infimum of lengths (wrt $g$) of all $C^{1}$ curves between $x$ and $y$. A \textit{geodesic} is a smooth curve on $\M$ that realizes locally this infimum. For $x \in \M$ and $v \in T_{x}\M$, there exists a unique geodesic $\gamma_{x,v}$ such that $\gamma_{x,v}(0)=x$ and $\dot{\gamma_{x,v}}(0)=v$. By the Hopf-Rinow theorem, the metric space $(\M,\rho_{g})$ is complete iff all geodesics are defined on $\R$, which we assume in the sequel. In that case, for all $x, y \in \M$, there exists at least one geodesic of minimal length between $x$ and $y$. For $x \in \M$, the \textit{exponential map} at $x$ is the map $\mathrm{Exp}_{x}^{\M} : T_{x}\M \lra \M$ such that $\Exp_{x}^{\M}(v) = \gamma_{x,v}(1)$. 

We recall that, for all $x \in \M$, any geodesic starting at $x$ is locally length-minimizing. The \textit{cut locus} of $x \in \M$ is the set $\Cut(x)$ of points $y$ such that the geodesics starting at $x$ cease to be length-minimizing beyond $y$. Thus, on a sphere, the geodesics are the great circles, so that two antipodal points are each other's cut locus. Let 
$x \in \M$ and $y \in \M \setminus \Cut(x)$. Then, there exists a \textit{unique} geodesic of minimal length between $x$ and $y$, whose initial velocity is thus the smallest tangent vector $v \in T_{x}\M$ such that $\Exp_{x}^{\M}\left(v\right) = y$. This vector is called the \textit{Riemannian Logarithm} of $y$ at $x$, denoted by $\mathrm{Log}_{x}^{\M}(y)$, related to the geodesic distance by:  
\begin{equation}\label{dgLog}
\rho_{g}(x,y) = \left\| \Log_{x}^{\M}(y) \right\|_{x}.
\end{equation}

\noindent
For any $x \in \M$, there exists an open set $\ID_{x} \subset T_{x}\M$, called the \textit{injectivity domain} at $x$, such that the restriction of $\Exp_{x}^{\M}$ to $\ID_{x}$ is a  diffeomorphism onto $\M \setminus \Cut(x)$, whose inverse is a restriction of the map $\Log_{x}^{\M}(\cdot)$. Finally, we present results on \textit{Riemannian submersions}.

\begin{definition}
Let $\left(\E, g^{\E} \right)$ and $\left(\B, g^{\B}\right)$ be Riemannian manifolds. Let $\pi : \E \lra \B$ be a smooth submersion, i.e. for all $x \in \E$, $d_{x}\pi$ is a surjective linear map. 

\vspace{.15cm}

\noindent
$(i)$ For $x \in \E$, the \textit{vertical space} at $x$ is $V_{x}\E := \ker (d_{x}\pi) \subset T_{x}\E$. The \textit{horizontal space} at $x$ is its orthogonal complement in $T_{x}\E$ wrt the inner product $g_{x}^{\E}$, denoted by $H_{x}\E := \left( V_{x}\E \right)^{\perp}$. 

\vspace{.15cm}

\noindent
$(ii)$ Let $y \in \B$ and $v \in T_{y}\B$. By $(i)$, for all $x \in \E$, there exists a unique vector $v^{\sharp}_{x} \in H_{x}$ such that $(d_{x}\pi)(v^{\sharp}_{x}) = v$. The vector  $v^{\sharp}_{x}$ is called the \textit{horizontal lift} wrt $\pi$ of $v$ at $x$.

\vspace{.15cm}

\noindent
$(iii)$ The map $\pi$ is a \textit{Riemannian submersion} when for all $x \in \E$, the restriction of $d_{x}\pi$ to $H_{x}\E$ is a linear isometry between $\left(H_{x}\E, g^{\E}_{x} \right)$ and $\left(T_{\pi(x)}\B, g^{\B}_{\pi(x)} \right)$. 

\end{definition}

\begin{proposition}\label{PropHorLift}
Let $\pi : \E \lra \B$ be a Riemannian submersion. 

\vspace{.15cm}

\noindent
$(i)$ Let $\gamma$ be a geodesic in $\B$. Then, for all $x \in \E$, there exists a unique geodesic $\hgm^{x}$ in $\E$ through $x$ which is horizontal, i.e. all its tangent vectors are horizontal, and which projects to $\gamma$, i.e. $\pi(\hgm^{x})=\gamma$. The geodesic $\hgm^{x}$ is called the horizontal lift wrt $\pi$ of $\gamma$ through $x$. 

\vspace{.15cm}

\noindent
$(ii)$ The geodesics in $\B$ are the images by $\pi$ of the horizontal geodesics in $\E$.
\end{proposition}

\subsubsection{Metrics on Stiefel manifolds and Grassmannians}

We recall hereafter the description of the tangent spaces of $\St_{q}$ and $G_{q}$: See \cite{Bendokat Zimmermann and Absil 2024}. For any $\cU \in \St_{q}$ and $P \in G_{q}$, 
\begin{equation}\label{tangentSpaceGrass}
T_{\cU}\St_{q} = \left\{ \cD \in \R^{d \times q} : \cU' \cD = - \cD'\cU \right\}
\enskip \textrm{and} \enskip
T_{P}G_{q} = \left\{ \Delta \in \mathrm{Sym}_{d} : \Delta P + P \Delta = \Delta \right\}. 
\end{equation}

\noindent
In the sequel, we endow $\St_{q}$ with the Euclidean metric $g^{\St}$ defined by
\begin{equation*}
g^{\St}_{\cU}\left( \cD_{1}, \cD_{2} \right)=\tr \left( \cD_{1}^{T} \cD_{2} \right), \quad \cU \in \St_{q} \textrm{~and~} \cD_{1}, \cD_{2} \in T_{\cU}\St_{q}. 
\end{equation*} 

\noindent
Then, $G_{q}$ is endowed with the metric $g^{G}$ (see \cite{Bendokat Zimmermann and Absil 2024} for a discussion on it) defined by
\begin{equation*}
g_{P}^{G}(\Delta_{1}, \Delta_{2}) = \frac{1}{2} \tr \left( \Delta_{1} \Delta_{2} \right), \quad 
P \in G_{q} \textrm{ and } \Delta_{1},  \Delta_{2} \in T_{P}G_{q}. 
\end{equation*}

\noindent
The Riemannian manifold $(G_{q}, g^{G})$ has a rich structure, notably of \textit{symmetric space}. So, many Riemannian operations in $G_{q}$ are available in closed form: see paragraph $\ref{explicitFormulas}$. Moreover, the following result holds: see \cite{Rabenoro and Pennec 2024}.

\begin{lemma}\label{piSG}
The map $\pi_{q}$ is a Riemannian submersion from $\left( \St_{q}, g^{\St} \right)$ onto $\left( G_{q}, g^{G} \right)$. 
\end{lemma}

\subsection{Geodesic parametrization of a full flag}\label{subGeoParam}

With the notations of Section $\ref{sec2}$, let $F^{\I}(\Sigma)$ be the flag of eigenspaces of $\Sigma$ ($\I$ denotes its type) and $\hF_{n}$ that of $\ECM$. Thus, $F^{\I}(\Sigma)$ is the flag of PS's and $\hF_{n}$ is a.s. a full flag. As in Anderson's paper \cite{Anderson 1963}, $\hF_{n}$ is represented by the matrix $C_{n} \in O(d)$ of eigenvectors of $\ECM$. When $F^{\I}(\Sigma)$ is estimated by $\hF_{n}$, the main difficulty to measure the deviation between them is that their types are different in general. 

To solve this problem, we define a parametrization of (almost) any full flag, represented by an orthogonal matrix, which determines its relative position wrt the pair formed by a given flag and an orthogonal matrix that spans the latter. Then, we prove that such a parametrization of $C_{n}$ wrt $\left( F^{\I}(\Sigma), \Gamma \right)$ is equivalent to that of $E_{n}:=\Gamma' C_{n}$ wrt $\left( \PoI, I_{d} \right)$.

\subsubsection{Geodesic projection of a frame}\label{geoProj}

Fix $R \in G_{q}$ and set 
\begin{equation*}
\V_{R} := \pi_{q}^{-1}\left( G_{q} \setminus \Cut(R) \right) = \left\{ \cU \in \St(q,d) : \pi_{q}(\cU) \notin \Cut(R) \right\}. 
\end{equation*}

\noindent
For $\cU \in \V_{R}$, set $P:=\pi_{q}(\cU)$. Let $\gamma_{(P,R)}$ be the minimal geodesic between $P$ and $R$ in $G_{q}$ and $\hgm_{(P,R)}^{\cU}$ its horizontal lift wrt 
$\pi_{q}$ through $\cU$. Then, define the $q$-frame $\cH^{q}_{(P,R)}(\cU)$ as the \textit{endpoint} of $\hgm_{(P,R)}^{\cU}$. By $(ii)$ of Proposition $\ref{PropHorLift}$, $\pi_{q} \left( \cH^{q}_{(P,R)}(\cU) \right) = R$. 

\[\xymatrix{
   \cU \ar@{.}[rr]_{\hgm_{(P,R)}^{\cU}} \ar[d]_{\pi_{q}} &&  
   \cH^{q}_{(P,R)}(\cU)  \ar[d]^{\pi_{q}} \\
    P \ar@{.}[rr]_{\gamma_{(P,R)}} && R 
  }
\]

\begin{remark}\label{lemarkMinGD}
Among all $q$-frames generating $\rg(R)$, $\cH^{q}_{(P, R)}(\cU)$ minimizes the geodesic distance in $\St_{q}$ to $\cU$, denoted by $d_{g}^{\St}$: See Lemma 26.11 in \cite{Michor 2008}. Formally, 
\begin{equation}\label{minGeoDist}
d_{g}^{\St} \left( \cU, \cH^{q}_{(P, R)}(\cU) \right) = \min \left\{ d_{g}^{\St} \left( \cU, \cW \right) : \pi^{SG}(\cW) = R \right\}.  
\end{equation}

\noindent
So, $\cH^{q}_{(P, R)}(\cU)$ is interpreted as the orthogonal projection of $\cU$ on $\rg(R)$ wrt the metric $g^{\St}$.
\end{remark}

\begin{definition}\label{defGeoProj}
The $q$-frame $\cH^{q}_{(P,R)}(\cU)$ is called the \textit{geodesic projection} of $\cU$ onto $\rg(R)$. For $v \in T_{P}G_{q}$, its horizontal lift wrt $\pi_{q}$ at $\cU$ is denoted by $v^{\sharp}_{\cU}$. By definition, 
\begin{equation}\label{formalGeoProj}
\cH^{q}_{(P,R)}(\cU) = \Exp^{\St_{q}}_{\cU} \left( \left( \Log^{G_{q}}_{P} \left( R \right) \right)^{\sharp}_{\cU} \right). 
\end{equation}
\end{definition}

\subsubsection{Geodesic decomposition of a frame}\label{deviationGeoDec}

Fix $R \in G_{q}$ and $\cW_{R} \in \pi_{q}^{-1}(R)$.

\begin{lemma}\label{lemmaGeoDec}
$(i)$ Define the map $\theta_{R} : \V_{R} \lra \pi_{q}\left( \V_{R} \right) \times \pi_{q}^{-1}(R)$ by 
\begin{equation}\label{frameGeoDec}
\theta_{R}(\cU) = \left( \pi_{q}(\cU), \cH^{q}_{(\pi_{q}(\cU), R)}(\cU) \right). 
\end{equation}

\noindent
Then, the map $\theta_{R}$ is a diffeomorphism, whose inverse is defined by 
\begin{equation*}
\theta_{R}^{-1} : \pi_{q}\left( \V_{R} \right) \times \pi_{q}^{-1}(R) \lra \V_{R} 
\quad \textrm{with} \quad 
(P, \cV) \lmp \cH^{q}_{(R, P)}(\cV). 
\end{equation*}

\noindent
$(ii)$ Let $\ID_{R}$ be the injectivity domain at $R$. Define the map $\xi_{R}^{\cW_{R}} : \V_{R} \lra \ID_{R} \times O(q)$ by 
\begin{equation}\label{effectiveGeoDec}
\xi_{R}^{\cW_{R}}(\cU) = \left( \Log_{R}^{G_{q}}\left(\pi_{q}(\cU) \right), (\cW_{R})' \cH_{(\pi_{q}(\cU), R)}(\cU) \right). 
\end{equation}

\noindent
Then, the map $\xi_{R}^{\cW_{R}}$ is a diffeomorphism. 
\end{lemma}

\begin{proof}
One checks that the map $\theta_{R}^{-1}$ is indeed the inverse of $\theta_{R}$: see the diagram in Definition $\ref{defGeoProj}$. Their smoothness follows from properties of Riemannian submersions. This proves $(i)$. Then, $(ii)$ follows readily from $(i)$ and Remark $\ref{fiberPiQ}$. 
\end{proof}

\begin{definition}
For $\cU \in \V_{R}$, the pair $\theta_{R}\left( \cU \right)$ is called the \textit{geodesic decomposition} of $\cU$ wrt $R$ and $\xi_{R}^{\cW_{R}}\left( \cU \right)$ its \textit{effective geodesic decomposition} wrt $(R, \cW_{R})$. By Lemma $\ref{lemmaGeoDec}$, any $\cU \in \V_{R}$ is uniquely determined by these decompositions.
\end{definition}

\begin{remark}
The effective geodesic decomposition lies in a product of a vector space and a group. Thus, it provides a handleable measure of the relative position of any $q$-frame $\cU \in \V_{R}$ wrt a given $q$-linear subspace $R$ and a $q$-frame spanning $\rg(R)$.  
\end{remark}

\subsubsection{Explicit formulas}\label{explicitFormulas}

The explicit formulas in $G_{q}$ provide those for the (effective) geodesic decomposition of a frame. First, we describe the cut locus in $G_{q}$: see \cite{Bendokat Zimmermann and Absil 2024}.

\begin{lemma}\label{CutGrass}
$(i)$ Let $P \in G_{q}$. Then, writing $P=YY'$, 
\begin{equation*}
\mathrm{Cut}(P) = \left\{ R=ZZ' \in G_{q} : \mathrm{rank}\left(Y'Z \right) < q \right\}.
\end{equation*}

\noindent
$(ii)$ Thus, the following symmetry holds: $R \in \Cut(P) \iff P \in \Cut(R)$. 

\noindent
$(iii)$ The cut locus is compatible with the action of $O(d)$ on $G_{q}$, i.e. for all $Q \in O(d)$, 
\begin{equation*}
R \in \Cut(P) \iff Q \cdot R \in \mathrm{Cut}(Q \cdot P). 
\end{equation*}
\end{lemma}

\noindent
Now, a closed form for the Riemannian Logarithm in $G_{q}$ is presented hereafter: see 
\cite{Batzies Huper Machado and Silva Leite 2015}.

\begin{theorem}\label{theoLogGrass}
Let $P \in G_{q}$ and $R \in G_{q} \setminus \Cut(P)$. Then, 
\begin{equation}\label{explicitLogGrass}
\mathrm{Log}_{P}^{G_{q}}(R) = [\Omega, P] \quad \textrm{where} \quad
\Omega = \frac{1}{2} \log_{m} \left( (I_{d}-2R)(I_{d}-2P) \right),
\end{equation}

\noindent
and $\log_{m}$ denotes the matrix logarithm. 
\end{theorem}

\begin{corollary}\label{O(d)andLogGrass}
Let $P, R \in G_{q}$ and $Q \in O(d)$. If $R \notin \Cut(P)$, then
\begin{equation*}
\Log_{Q \cdot P}^{G_{q}}\left( Q \cdot R \right) = Q \left( \Log_{P}^{G_{q}}\left( R \right) \right) Q'.
\end{equation*}
\end{corollary}

\begin{proof}
By the properties of the matrix logarithm, 
\begin{align*}
\log_{m} \left( (I_{d}-2QRQ')(I_{d}-2QPQ') \right) &= \log_{m} \left(Q (I_{d}-2R)(I_{d}-2P) Q' \right) \\
&= Q\log_{m} \left( (I_{d}-2R)(I_{d}-2P) \right)Q'.
\end{align*}

\noindent
So, with the notations of $(\ref{explicitLogGrass})$, $\Log_{Q \cdot P}^{G_{q}}\left( Q \cdot R \right) =\left[ Q\Omega Q', QPQ' \right] = Q[\Omega, P]Q'$.
\end{proof}

\noindent
Then, we establish from $(\ref{formalGeoProj})$ an explicit expression of $\cH^{q}_{(P, R)}(\cU)$, proved in \cite{Rabenoro and Pennec 2024}. 

\begin{proposition}\label{holSG}
For $\cU \in \V_{R}$, set $P:=\pi_{q}(\cU)$ and $\Delta := \Log_{P}^{G_{q}} \left(R\right)$. Then, setting also 
$\cV := \begin{pmatrix}
\cU & \Delta \cU 
\end{pmatrix} \in \R^{n \times 2q}$ 
and $C:=\cU' \Delta^{2}\cU$,
\begin{equation}\label{equHolSG}
\cH^{q}_{(P, R)}(\cU) = \cV \left( \exp_{m} \begin{pmatrix}
0 & -C \\
I_{q} & 0
\end{pmatrix}
\right) \begin{pmatrix}
I_{q} \\ 0_{q} 
\end{pmatrix}.
\end{equation}

\end{proposition}

\noindent
We deduce hereafter the effect on the geodesic decomposition of the action of $O(d)$.

\begin{corollary}\label{geoDecO(d)}
$(i)$ Let $R \in G_{q}$ and $\cU \in \V_{R}$. Then, for all $Q \in O(d)$, 
\begin{equation}\label{geoDecQLog}
Q\cU \in \V_{Q \cdot R}
\quad \textrm{and} \quad
\Log_{Q \cdot R}^{G_{q}}\left(\pi_{q}(Q\cU) \right) = Q \left( \Log_{R}^{G_{q}}\left(\pi_{q}(\cU) \right) \right) Q'.
\end{equation}

\noindent
$(ii)$ For the geodesic projections, we have that 
\begin{equation*}
\cH_{(\pi_{q}(Q\cU), Q \cdot R)}(Q\cU) = Q \cH_{(\pi_{q}(\cU), R)}(\cU). 
\end{equation*}

\noindent
$(iii)$ For any $\cW_{R} \in \pi_{q}^{-1}(R)$, the effective geodesic decomposition of $\cU$ wrt $(R, \cW_{R})$ is fully determined by that of $Q\cU$ wrt $(Q\cdot R, Q\cW_{R})$.
\end{corollary}

\begin{proof}
$(i)$ and $(ii)$ follow from Lemma $\ref{CutGrass}$, Corollary $\ref{O(d)andLogGrass}$ and Proposition $\ref{holSG}$. Then, $(iii)$ is deduced from $(i)$ and $(ii)$. 
\end{proof}

\subsubsection{Geodesic parametrization of an orthogonal matrix}\label{geoParamO(d)}

In this paragraph, we consider a fixed flag $\cR=(R_{i})_{i} \in \FI$ and a matrix $\Gamma \in O(d)$ which spans $\cR$. Now, set
\begin{equation}\label{defWR}
\W_{\cR} := \left\{ Q \in O(d) : \forall~ 1 \leq i \leq r, ~\pi^{i}(Q^{(i)}) \notin \Cut(R_{i}) \textrm{ i.e. } Q \Poi Q' \notin \Cut(R_{i}) \right\}.  
\end{equation}

\noindent
For $C \in \W_{\cR}$ identified with $\zeta(C) := \left( C^{(i)} \right)_{i}$, consider the collection of geodesic decompositions of $C^{(i)}$ wrt $R_{i}$, illustrated hereafter, where 
$(\ref{linkPiSproof})$ justifies that $\pii(C) = \left( \pi^{i}\left( C^{(i)}\right) \right)_{i}$. 

\[
\xymatrix{
    C \ar[rr]_{\zeta} \ar[rrd]_{\pii} && \zeta(C) := \left( C^{(i)}\right)_{i} \ar[d]^{(\pi^{i})_{i}} \ar@{.>}[rr]  && \left( \cH^{q_{i}}_{(\pi^{i}\left( C^{(i)}\right), R_{i})}\left( C^{(i)} \right) \right)_{i} \ar[d]^{(\pi^{i})_{i}} \\
    && \pii(C) = \left( \pi^{i}\left( C^{(i)}\right) \right)_{i} \ar@{.>}[rr] && \cR = (R_{i})_{i}
    }
\]

\begin{definition}
The \textit{geodesic parametrization} of $C \in \W_{\cR}$ (or of the associated full flag) wrt $(\cR, \Gamma)$ is defined as the collection of \textit{effective} geodesic decompositions of $C^{(i)}$ wrt $\left(R_{i}, \Gamma^{(i)} \right)$, for $1 \leq i \leq r$. This parametrization determines $C \in \W_{\cR}$ uniquely. 
\end{definition}

\begin{lemma}\label{lemGeoDecE}
For $C \in \W_{\cR}$, set $E:=\Gamma' C$. Then, $E \in \W_{\PoI}$ and the geodesic parametrization of $C$ wrt $(\cR, \Gamma)$ is fully determined by that of $E$ wrt $(\PoI , I_{d})$. 
\end{lemma}

\begin{proof}
We apply Corollary $\ref{geoDecO(d)}$ with $Q=\Gamma'$. Thus, the geodesic parametrization of $C$ wrt $(\cR, \Gamma)$ is fully determined by that of $E:=\Gamma'C$ wrt $(\Gamma' * \cR, \Gamma' \Gamma) = ( \Gamma^{-1} * \cR , I_{d})$. Now, $\Gamma \in O(d)$ spans the flag $\cR$. So, by $(\ref{remarkSpanFlag})$, $\cR=\Gamma*\PoI$, which implies that $\Gamma^{-1} * \cR = \PoI$. 
\end{proof}

\subsubsection{Tangent spaces at standard Grassmannians}

A geodesic parametrization wrt $(\PoI , I_{d})$ involves the maps $\Log^{G^{i}}_{\Poi}$, for $1 \leq i \leq r$, each valued in the tangent space $T_{\Poi} G^{i}$. Thus, we describe hereafter the structure of the latter.

\begin{lemma}\label{visualToi}
For $1 \leq i \leq r$, set $\Toi := T_{\Poi} G^{i}$. Then,  
\begin{equation*}
\T_{0}^{1} = \left\{
\begin{pmatrix}
0_{q_{1}} & A^{1}_{>} \\
\left(A^{1}_{>}\right)' & 0 \\
\end{pmatrix} : 
A^{1}_{>} \in \R^{q_{1} \times (q_{2}+...+q_{r})}
\right\}
~ , ~ 
\T_{0}^{r} = \left\{
\begin{pmatrix}
0 & \left(A^{r}_{<} \right)' \\
A^{r}_{<} & 0_{q_{r}} \\
\end{pmatrix} : 
A^{r}_{<} \in \R^{q_{r} \times (q_{1}+...+q_{r-1})}
\right\}
\end{equation*}

\begin{equation*}
\Toi = \left\{
\begin{pmatrix}
0 & \left(A^{i}_{<}\right)' & 0 \\
A^{i}_{<} & 0_{q_{i}} & A^{i}_{>} \\
0 & \left(A^{i}_{>}\right)' & 0
\end{pmatrix} : 
A^{i}_{<} \in \R^{q_{i} \times (q_{1}+...+q_{i-1})} ~,~ A^{i}_{>} \in \R^{q_{i} \times (q_{i+1}+...+q_{r})}
\right\}, \quad i \notin \left\{1, r\right\}.
\end{equation*}

\end{lemma}

\begin{proof}
By $(\ref{tangentSpaceGrass})$, for any $P \in G(q,d)$, $T_{P}G(q,d) = \left\{ \Delta \in \mathrm{Sym}_{d} : \Delta P + P \Delta = \Delta \right\}$. So, 
\begin{equation}\label{conditionToi}
\Delta \in \Toi \iff \Delta \in \Sym \textrm{ and } \Delta\Poi + \Poi\Delta = \Delta. 
\end{equation}

\noindent
We readily check that the sets described in the rhs' in the statement of this Lemma satisfy $(\ref{conditionToi})$. Thus, each of these sets is a linear subspace of $\Toi$, and clearly, its  dimension is equal to that of $\Toi$. This concludes the proof. 
\end{proof}

\begin{corollary}\label{coroSumSquares}
For any $d \times d$ matrix $A=(a_{k\ell})_{1 \leq k \leq \ell \leq d}$, its Frobenius norm is 
\begin{equation}\label{defFrobNorm}
\left\| A \right\|_{F} := \tr(A'A) = \sqrt{\sum (a_{k\ell})^{2}}. 
\end{equation}

\noindent
Let $1 \leq i \leq r$. Then, for all $\cS \in \Toi$, 
\begin{equation}\label{sumSquaresProj}
\sum\limits_{i=1}^{r} \left\| \cS \right\|_{F}^{2} = 
4 \sum\limits_{1 \leq i < j \leq r} \left\| \cS^{(i,j)} \right\|_{F}^{2}.
\end{equation}
\end{corollary}

\begin{proof}
This identity is an easy consequence of Lemma $\ref{visualToi}$ and $(\ref{defFrobNorm})$. 
\end{proof}

\subsection{Preliminaries for estimating the flag of PS's}\label{discrepancyFlags}

The flag $F^{\I}(\Sigma)$ of PS's is estimated by the \textit{flag of eigenprojections} $F^{\I}\left( \ECM \right)$, of type $\I$, introduced in paragraph $\ref{flagEigenproj}$ below. Thus, we need a distance on $\FI$ to measure the deviation between flags of type $\I$. However, there is no known metric on $\FI$ for which the geodesic distance on $\FI$ is available in closed form. To overcome this lack, we introduce in $(\ref{defDex})$ below, an \textit{extrinsic distance} $\fD_{\ex}$ on $\FI$ i.e. that is the restriction to $\FI$ of a distance on $\GI$, which has an explicit expression.

\subsubsection{Flag of eigenprojections}\label{flagEigenproj}

Let $\Sigma \in \Sym$, whose flag of eigenspaces, denoted by $F^{\I}(\Sigma)$, is of type $\I := (q_{i})_{1 \leq i \leq r}$. For $S \in \Symdiff$, let $\Gamma$ and $C$ be respective orthogonal matrices of eigenvectors of $\Sigma$ and $S$, associated to eigenvalues in decreasing order.

For all $1 \leq i \leq r$, $\pi^{i}\left(C^{(i)} \right)$ is the projector onto a $q_{i}$-linear subspace spanned by eigenvectors of $C$. This projector is independent of the choice of $C$ i.e. depends only on $S$. Thus it is called the $i$-th \textit{eigenprojection} of $S$ wrt $\I$, denoted by $P_{i}(S)$. Then, the collection $F^{\I}(S):=\left( P_{i}(S) \right)_{i}$ form a flag, called the \textit{flag of eigenprojections} of $S$ wrt $\I$. By $(\ref{linkPiSproof})$, 
\begin{equation}\label{FISequalPiiC}
F^{\I}(S) = \left( P_{i}(S) \right)_{i} = \left( \pi^{i}\left(C^{(i)} \right) \right)_{i} = \pii(C).
\end{equation}

\noindent
Assume that $C \in \W_{F^{\I}(\Sigma)}$. Set $E:=\Gamma' C$. Since $\Gamma$ spans $F^{\I}(\Sigma)$, by Lemma $\ref{lemGeoDecE}$, the geodesic paramatrization of $C$ wrt $\left( F^{\I}\left( \Sigma \right) , \Gamma \right)$ is fully determined by that of $E$ wrt $\left( \PoI , I_{d} \right)$, for which 
\begin{equation}\label{piiOfE(i)}
\pi^{i}\left(E^{(i)}\right) = \Gamma' \cdot \pi^{i}\left(C^{(i)} \right) = \Gamma' \cdot P_{i}(S), \quad 1 \leq i \leq r.
\end{equation}

\subsubsection{Extrinsic distance}
Let $\cR=(R_{i})_{i} \in \FI$ and $\W_{\cR}$ defined in $(\ref{defWR})$. Then, 
\begin{equation*}
\pii \left( \W_{\cR} \right) = \left\{ \cP=(P_{i})_{i} \in \FI : \forall 1 \leq i \leq r, ~P_{i} \notin \Cut(R_{i}) \right\}. 
\end{equation*}

\noindent
Now, for $\cP \in \pii \left( \W_{\cR} \right)$, set  
\begin{equation}\label{defDex}
\fD_{\ex}(\cP, \cR) := \left( \sum\limits_{i=1}^{r} \delta^{i}(P_{i}, R_{i})^{2} \right)^{1/2} 
\enskip \textrm{where} \enskip
\delta^{i}(P_{i}, R_{i}) := \left\| \Log_{P_{i}}^{G^{i}} \left(R_{i} \right) \right\|_{F}.
\end{equation}

\noindent
Thus, $\delta^{i}$ is the geodesic distance on $G^{i}$. The distance $\fD_{\ex}$ is \textit{extrinsic}, i.e. it is the restriction to $\FI$ of a distance on $\GI$, and has an \textit{explicit expression}, provided by that of $\Log_{P_{i}}^{G^{i}}$. Clearly, the action of $O(d)$ on $\FI$ preserves $\fD_{ex}$, i.e. for all $Q \in O(d)$, 
\begin{equation}\label{OdPreservesDex}
\fD_{\ex}(Q*\cP, Q*\cR) = \fD_{\ex}(\cP, \cR). 
\end{equation}

\noindent
In particular, with the notations of paragraph $\ref{geoParamO(d)}$, we deduce from $(\ref{OdPreservesDex})$ that for $E:=\Gamma' C$, 
\begin{equation*}
\fD_{ex}\left( \pii(E) , \PoI \right) = \fD_{ex}\left( \Gamma' *\pii(C) , \Gamma' *F^{\I}(\Sigma) \right) = \fD_{ex}\left( F^{\I}(S) , F^{\I}(\Sigma) \right).
\end{equation*}

\subsubsection{A discrepancy between flags}

In Section $\ref{sec4}$, to obtain a pivotal statistic for the flag $F^{\I}(\Sigma)$ of PS's, we need to standardize some Gaussian rv's. Due to such renormalizations, the resulting confidence regions for $F^{\I}(\Sigma)$ are expressed in terms of a \textit{discrepancy} between $F^{\I}(\Sigma)$ and $F^{\I}(\ECM)$, instead of the extrinsic distance $\fD_{ex}$ between them. In fact, this discrepancy, defined in $(\ref{discrepancyQGamma})$ below, is a \textit{deformation} of $\fD_{ex}$. This means that, without these renormalizations, this discrepancy is equal to $\fD_{ex}$. Let $\cD(\I)$ be the set of all block-diagonal matrices whose structure agrees with the type $\I$, i.e. 
\begin{equation*}
M \in \cD(\I) \iff M=\Dg \left( M^{(1,1)}, ..., M^{(i,i)}, ..., M^{(r,r)} \right). 
\end{equation*}

\noindent
\textit{In this paragraph, $\Gamma$ denotes a variable matrix in $O(d)$}. Let $\K=\left(K^{i}\right)_{1 \leq i \leq r}$ be a collection of matrices such that for all $1 \leq i \leq r$, $K^{i} \in \cD(\I)$.  Let $\cP=(P_{i})_{i} \in \FI$. By Lemma $\ref{CutGrass}$, 
\begin{equation}\label{GammaInWP}
\Gamma \in \W_{\cP} \iff \forall~ 1 \leq i \leq r, ~\Gamma' P_{i}\Gamma \notin \Cut(\Poi). 
\end{equation}

\noindent
Then, for $\Gamma \in \W_{\cP}$, set 
\begin{equation}\label{discrepancyQGamma}
\fD_{\K}\left( \Gamma, \cP \right) := \sqrt{\sum\limits_{i=1}^{r} \left( \delta_{K^{i}}^{i}\left( \Gamma, P_{i}\right) \right)^{2}} 
\enskip \textrm{where} \enskip
\delta_{K^{i}}^{i}(\Gamma, P_{i}) := \left\| K^{i} \Logi \left(\Gamma' \cdot P_{i} \right) K^{i} \right\|_{F}.
\end{equation}

\begin{lemma}\label{InvarianceDiscrepancy}
Let $\cP=(P_{i})_{i} \in \FI$ and $\Gamma \in \W_{\cP}$. Then, $\fD_{\K}\left( \Gamma, \cP \right)$ is independent of the class of $\Gamma$ modulo $O(\I)$, i.e. for all $H \in O(\I)$, \begin{equation}\label{invaOF}
\Gamma H \in \W_{\cP}
\quad \textrm{and} \quad
\fD_{\K}\left( \Gamma H, \cP \right) = \fD_{\K}\left( \Gamma, \cP \right). 
\end{equation}
\end{lemma}

\begin{proof}
Recall that $O(\I)$ is the isotropy group of $\PoI$ under the action of $O(d)$ on $\FI$, i.e. for all $H \in O(\I)$ and $1 \leq i \leq r$, $H \cdot \Poi = \Poi$. So, by $(\ref{GammaInWP})$,  for all $H \in O(\I)$, 
\begin{equation}\label{invarianceCut}
\Gamma \in \W_{\cP} \iff \Gamma H \in \W_{\cP}. 
\end{equation}

\noindent
Then, by Corollary $\ref{O(d)andLogGrass}$, for any $\cP=(P_{i})_{i} \in \FI$, $\Gamma \in \W_{\cP}$ and $H \in O(\I)$, 
\begin{equation*}
\Logi \left( (\Gamma H)' \cdot R_{i} \right) = H' \left( \Logi \left( \Gamma' \cdot R_{i} \right) \right)H.
\end{equation*}

\noindent
This, combined to Remark $\ref{commHK}$ below, implies that for all $1 \leq i \leq r$, 
\begin{equation*}
\delta_{K^{i}}^{i}(\Gamma H, P_{i}) = 
\left\| K^{i} \left(H' \Log_{\Poi}^{G^{i}} \left( \Gamma' \cdot P_{i} \right) H\right) K^{i} \right\|_{F} = 
\left\| H' \left(K^{i} \Log_{\Poi}^{G^{i}}\left(\Gamma' \cdot P_{i} \right) K^{i} \right) H \right\|_{F}.
\end{equation*}

\noindent
Finally, by the properties of $\left\| \cdot \right\|_{F}$, the rhs hereabove is equal to $\delta_{K^{i}}^{i}(\Gamma, P_{i})$. 
\end{proof}

\begin{remark}\label{commHK}
For all $H \in O(\I)$ and $K \in \mathcal{D}(\I)$, $H$ and $K$ commute, i.e. $HK=KH$. 
\end{remark}

\noindent
Now, we may introduce a discrepancy between flags of type $\I$ as follows. 

\begin{definition}\label{defKdiscrepancy}
Let $\K=\left(K^{i}\right)_{i} \in \left( \cD(\I) \right)^{r}$. Let $\cP, \cR \in \FI$ such that $\cR \in \pii\left( \W_{\cP} \right)$. Then, by Lemma $\ref{InvarianceDiscrepancy}$, we may define the $\K$-discrepancy between $\cR$ and $\cP$ as follows.   
\begin{equation}\label{defDiscrepancyFlags}
\widetilde{\fD}_{\K}\left( \cR, \cP \right) := \fD_{\K}\left( \Gamma, \cP \right), 
\end{equation}

\noindent
for any $\Gamma \in O(d)$ such that $\pii(\Gamma)=\cR$, which is defined modulo $O(\I)$. In particular, it is easy to check that, for $\mathbb{I} := (I_{d}, ..., I_{d}) \in \left( \cD(\I) \right)^{r}$, we have that $\widetilde{\fD}_{\mathbb{I}}\left( \cR, \cP \right) = \fD_{\ex}\left( \cR, \cP \right)$. Thus, the $\K$-discrepancy $\widetilde{\fD}_{\K}$ is interpreted as a \textit{deformation} of the extrinsic distance $\fD_{\ex}$. 
\end{definition}

\begin{remark}
In general, $\widetilde{\fD}_{\K}$ is not a distance on $\FI$. However, we readily prove that $\widetilde{\fD}_{\K}$ fulfills the following properties. 
$(i)$ The separation property holds, i.e. for all $\cR, \cP \in \FI$
\begin{equation}\label{separationDiscrepancy}
\widetilde{\fD}_{\K}\left( \cR, \cP \right) = 0 \iff \cR = \cP. 
\end{equation}

\noindent
$(ii)$ The $\K$-discrepancy preserves the action of $O(d)$, i.e. for all $Q \in O(d)$,
\begin{equation*}
\widetilde{\fD}_{\K}\left( Q*\cR, Q*\cP \right) = \widetilde{\fD}_{\K}\left( \cR, \cP \right), \quad \cR, \cP \in \FI. 
\end{equation*}
\end{remark}

\section{Limit theorem and estimation of the flag of PS's}\label{sec4}

If $E_{n} \in \W_{\PoI}$, then we establish in Theorem $\ref{theoCLTgrass}$ the limiting distribution of its geodesic parametrization wrt $\left( \PoI, I_{d} \right)$. In fact, $\Pr\left( E_{n} \notin \W_{\PoI} \right) \xrightarrow[]{} 0$ as $n \ra \infty$: see $(i)$ of Proposition $\ref{PropUnSniStar}$ in Appendix. Thus, we obtain Gaussian and Haar limiting distributions, as in Theorem $\ref{MainAnderson}$, where $E_{n}$ is parametrized by its usual coordinates. Then, we derive from $(i)$ of Theorem $\ref{theoCLTgrass}$, confidence regions 
for the flag $F^{\I}(\Sigma)$ of PS's, in terms of a discrepancy (in the sense of Definition $\ref{defKdiscrepancy}$) between $F^{\I}(\Sigma)$ and the flag $F^{\I}\left(\ECM \right)$ of \textit{eigenprojections} of $\ECM$, both of type $\I$.

\subsection{Statement of main theorem}

First, for $1 \leq i \leq r$, we define the limit function $g_{0}^{i}$ valued in $\Toi$ involved in $(\ref{CLTgrass})$ below. For $u \in \Sym$, set 
\begin{equation*}
\cG^{i,j}(u) := \frac{1}{\lambda_{i}-\lambda_{j}}u^{(i,j)}, \quad 1 \leq i, j \leq r \textrm{ with } i \neq j.
\end{equation*}

\noindent
Then, consider the matrices $\cG^{i}_{<}(u)$ and $\cG^{i}_{>}(u)$ of respective sizes $q_{i} \times (q_{1}+...+q_{i-1})$ and $q_{i} \times (q_{i+1}+...+q_{r})$ defined by 
\begin{equation*}
\cG^{i}_{<}(u) := 
\begin{pmatrix}
\cG^{i,1}(u) & ... & \cG^{i,i-1}(u)
\end{pmatrix}
\quad \textrm{and}\quad
\cG^{i}_{>}(u) := 
\begin{pmatrix}
\cG^{i,i+1}(u) & ... & \cG^{i,r}(u)
\end{pmatrix}
\end{equation*}

\noindent
Finally, define the map $g_{0}^{i} : \Sym \lra \Toi$ by 
\begin{equation*}
g_{0}^{1}(u) = 
\begin{pmatrix}
0_{q_{1}} & \cG^{1}_{>}(u) \\
\left(\cG^{1}_{>}(u)\right)' & 0 \\
\end{pmatrix}
\qquad , \qquad 
g_{0}^{r}(u) = 
\begin{pmatrix}
0 & \left(\cG^{r}_{<}(u) \right)' \\
\cG^{r}_{<}(u) & 0_{q_{r}} \\
\end{pmatrix}
\end{equation*}

\begin{equation*}
g_{0}^{i}(u) =
\begin{pmatrix}
0 & \left(\cG^{i}_{<}(u)\right)' & 0 \\
\cG^{i}_{<}(u) & 0_{q_{i}} & \cG^{i}_{>}(u) \\
0 & \left(\cG^{i}_{>}(u)\right)' & 0
\end{pmatrix}, 
\quad i \neq 1, r. 
\end{equation*}

\noindent
Now, we can state our main result hereafter, where the rv $U$ is defined in Teorem $\ref{CLTsym}$.

\begin{theorem}\label{theoCLTgrass}
$(i)$ For any $1 \leq i \leq r$, the following CLT holds in $\Toi$. 
\begin{equation}\label{CLTgrass}
\Gni := \sqrt{n}\Logi \left( \pi^{i}\left( E_{n}^{(i)} \right) \right) \xrightarrow[n \rightarrow \infty]{d} \goi(U). 
\end{equation}

\noindent
$(ii)$ For any $1 \leq i \leq r$, the following convergence holds in $O(q_{i})$. 
\begin{equation}\label{TheoCVtoHaar}
\Hni := \left( I_{d}^{(i)} \right)' \left( \cH^{q_{i}}_{\left( \pi^{i}\left( E_{n}^{(i)} \right) , P_{0}^{i} \right)} \left( E_{n}^{(i)} \right) \right) 
\xrightarrow[n \rightarrow \infty]{d} \cE^{i,i}(U), 
\end{equation}

\noindent
where the distribution of $\cE^{i,i}(U)$ is the conditional Haar invariant distribution. 
\end{theorem}

\begin{remark}\label{mineVsAnderson}
In conclusion, we conjecture a possible application of $(\ref{TheoCVtoHaar})$. 
\end{remark}

\subsection{Convergence to a $\chi^{2}$ distribution}\label{paragraphKhi2}

In Corollary $\ref{corollaryPivotal}$, we derive a \textit{pivotal statistic} for the matrix $\Gamma$, i.e. which depends only on the sample, and whose limit distribution does not depend on any unknown parameter. Thus, Lemma $\ref{towardsChiSquare}$ allows to \textit{concatenate the CLT's} of $(\ref{CLTgrass})$ for $1 \leq i \leq r$, providing the convergence of a pivotal statistic to a $\chi^{2}$ distribution. Namely, by Theorem \ref{CLTsym}, for $i \neq j$, the entries of $\left(g_{0}^{i}(U) \right)^{(i,j)} = \frac{1}{\lambda_{i}-\lambda_{j}}U^{(i,j)}$ are real i.i.d rv's $\cN(0,\sigma_{i,j}^{2})$, of standard deviation $\sigma_{i,j} = \frac{\sqrt{\lambda_{i}\lambda_{j}}}{|\lambda_{i}-\lambda_{j}|}$. We normalize these entries by setting
\begin{equation}\label{gBarKi}
\overline{g}_{0}^{i}(U) := K^{i}g^{i}(U)K^{i} \quad \textrm{where} \quad 
K^{i}:=\Dg\left( \frac{1}{\sigma_{i,1}}I_{q_{1}}~,~ ... ~,~I_{q_{i}} ~,~ ... ~,~ \frac{1}{\sigma_{i,r}}I_{q_{r}} \right).
\end{equation}

\begin{lemma}\label{towardsChiSquare}
$(i)$ The non-null entries of the matrix $\overline{g}_{0}^{i}(U)$ are real iid rv's $\mathcal{N}(0,1)$.

$(ii)$ The blocks $\left\{ \left( \overline{g}_{0}^{i}(U) \right)^{(i,j)} : 1 \leq i < j \leq r \right\}$ are mutually independent.

$(iii)$ The real rv $\frac{1}{4} \sum\limits_{i=1}^{r} \left\| \overline{g}_{0}^{i}(U) \right\|_{F}^{2}$ is distributed as a $\chi^{2}_{\mathrm{D}^{\I}}$, where 
\begin{equation*}
\mathrm{D}^{\I} := \frac{1}{2}\left( d^{2}-\sum\limits_{i=1}^{r} q_{i}^{2} \right). 
\end{equation*}
\end{lemma}

\begin{proof}
Clearly, $\overline{g}_{0}^{i}(U)$ is a random matrix valued in $\Toi$ and for $i \neq j$, 
\begin{equation}\label{goiBar}
\left( \overline{g}_{0}^{i}(U) \right)^{(i,j)} = \frac{1}{\sigma_{i,j}} \left( \frac{1}{\lambda_{i}-\lambda_{j}}U^{(i,j)} \right).     
\end{equation}

\noindent
This proves $(i)$ and $(ii)$. Then, $(iii)$ holds, by Corollary $\ref{coroSumSquares}$ applied with $\cS=\overline{g}_{0}^{i}(U)$.
\end{proof}

\begin{proposition}\label{convergenceChi2}
For $(K^{i})_{1 \leq i \leq r}$ defined in $(\ref{gBarKi})$, 
\begin{equation}\label{mergeCLTs}
\frac{n}{4} \sum\limits_{i=1}^{r} \left\| K^{i} \Logi \left(\Gamma' \cdot P_{i}\left( \hat{\Sigma}_{n}\right) \right) K^{i} \right\|_{F}^{2} \xrightarrow[n \rightarrow \infty]{d} \chi^{2}_{\mathrm{D}^{\I}}, 
\end{equation}
\end{proposition} 

\begin{proof}
By $(i)$ of Theorem $\ref{theoCLTgrass}$ and $(iii)$ of Lemma $\ref{towardsChiSquare}$, 
\begin{equation}\label{proofCVchi2}
\frac{n}{4} \sum\limits_{i=1}^{r} \left\| K^{i} \Logi \left( \pi^{i}\left( E_{n}^{(i)} \right) \right) K^{i} \right\|_{F}^{2} \xrightarrow[n \rightarrow \infty]{d} \frac{1}{4} \sum\limits_{i=1}^{r} \left\| \overline{g}_{0}^{i}(U) \right\|_{F}^{2} \sim \chi^{2}_{\mathrm{D}^{\I}}. 
\end{equation}

\noindent
Now, by $(\ref{piiOfE(i)})$ applied to $S = \ECM \in \Symdiff$ a.s., we have that $\pi^{i}\left( E_{n}^{(i)} \right) = \Gamma' \cdot P_{i}\left( \hat{\Sigma}_{n}\right)$ a.s. So, by injecting this relation in the lhs of $(\ref{proofCVchi2})$, we deduce that $(\ref{mergeCLTs})$ holds.  
\end{proof}

However, in the lhs of $(\ref{mergeCLTs})$, the matrix $K^{i}$ still depends on the unknown parameters $(\lambda_{i})_{i}$. Thus, in $K_{i}$, we replace $\lambda_{i}$ by $\widehat{\lambda}_{n}^{i}$, where (with the notations of paragraph 
$\ref{preliminariesAnderson}$), denoting by $\left(\mu_{k}\left(\ECM \right) \right)_{1 \leq k \leq d}$ the eigenvalues of $\ECM$ and by $(\beta_{i})_{i}$ the partition of $\left\{ 1, ..., d \right\}$ wrt $\I$, 
\begin{equation*}
\widehat{\lambda}_{n}^{i} := \frac{1}{q_{i}}\sum\limits_{k \in \beta_{i}} \mu_{k}\left(\ECM \right) \xrightarrow[n \rightarrow \infty]{a.s.} \lambda_{i}. 
\end{equation*}

\begin{corollary}\label{corollaryPivotal}
The statistic $\hS_{n}$ defined hereafter is a pivotal statistic for $\Gamma$. 
\begin{equation}\label{PivotStat}
\hS_{n} := \frac{n}{4} \sum\limits_{i=1}^{r}
\left\| \hK^{i}_{n} \Logi \left( \Gamma' \cdot P_{i}\left(\ECM \right) \right) \hK^{i}_{n} \right\|_{F}^{2} \xrightarrow[n \rightarrow \infty]{d} \chi^{2}_{\mathrm{D}^{\I}}, 
\end{equation}

\noindent
where, for $1 \leq i \leq r$, $\hK^{i}_{n}$ is derived from $K^{i}$ by replacing $\lambda_{i}$ by $\widehat{\lambda}_{n}^{i}$. 
\end{corollary}

\begin{proof}
Since $\hK^{i}_{n} \lra K^{i}$ a.s. as $n \ra \infty$, this follows from Proposition $\ref{convergenceChi2}$. 
\end{proof}

\subsection{Confidence regions}\label{paragraphCR}

For $n \geq 1$, set $\widehat{\K}_{n}:=\left(\widehat{K}^{i}_{n} \right)_{1 \leq i \leq r}$. Thus, by definition,  
\begin{equation}\label{avoidTruncationsTn}
\hS_{n} = \frac{n}{4} \left[ \fD_{\widehat{\K}_{n}}\left( \Gamma, F^{\I}\left( \ECM \right) \right) \right]^{2}, 
\end{equation}

\noindent
where $\fD_{\widehat{\K}_{n}}$ is defined in $(\ref{discrepancyQGamma})$. Thus, $\widetilde{\fD}_{\widehat{\K}_{n}}$ is the $\widehat{\K}_{n}$-discrepancy, in the sense of Definition $\ref{defKdiscrepancy}$. Since $\pii(\Gamma) = F^{\I}(\Sigma)$, Corollary $\ref{corollaryPivotal}$ and $(\ref{avoidTruncationsTn})$ imply that 
\begin{equation}\label{discrepancyCVinD}
\hS_{n} = \frac{n}{4} \left[ \widetilde{\fD}_{\widehat{\K}_{n}} \left( F^{\I}\left(\Sigma \right), F^{\I}\left(\ECM \right) \right) \right]^{2} \xrightarrow[n \rightarrow \infty]{d} \chi^{2}_{\mathrm{D}^{\I}}. 
\end{equation}

\noindent
We deduce hereafter 
confidence regions
for $F^{\I}\left( \Sigma \right)$ of the desired form.

\begin{proposition}\label{ConfidenceRegion}
For any $\alpha \in (0,1)$, let $\chi^{2}_{\mathrm{D}^{\I}}(1-\alpha)$ be the quantile of order $(1-\alpha)$ of the $\chi^{2}_{\mathrm{D}^{\I}}$ distribution. For $n \geq 1$, set 
\begin{equation*}
R_{n,\alpha} := \left\{ \cP \in \FI ~:~ \frac{n}{4} \left[ \widetilde{\fD}_{\widehat{\K}_{n}}\left( \cP , F^{\I}\left( \ECM \right) \right) \right]^{2} \leq \chi^{2}_{\mathrm{D}^{\I}}(1-\alpha) \right\}.
\end{equation*}

\noindent 
Then, by $(\ref{discrepancyCVinD})$, $R_{n,\alpha}$ is a 
confidence region
for $F^{\I}\left( \Sigma \right)$ of asymptotic level $(1-\alpha)$. By $(\ref{separationDiscrepancy})$, 
\begin{equation*}
\widetilde{\fD}_{\widehat{\K}_{n}}\left( \cP , F^{\I}\left( \ECM \right) \right) = 0 \iff Q = F^{\I}\left( \ECM \right). 
\end{equation*}

\noindent
So, $R_{n,\alpha}$ is interpreted as a full deformed ellipsoid in $\FI$, whose center is $F^{\I}\left( \ECM \right)$. 
\end{proposition}

\begin{corollary}
Fix $Q_{0} \in O(d)$. Consider the following null hypothesis assumption. $H_{0}$ : $\pi^{\mathrm{I}}\left(Q_{0}\right)=F^{\mathrm{I}}\left(\Sigma\right)$. For any $\alpha \in (0,1)$, consider the test which accepts $H_{0}$ when $\pii(Q_{0}) \in R_{n,\alpha}$ and rejects $H_{0}$ else. Then, this test is of asymptotic level $\alpha$.
\end{corollary}

\subsection{Simulation}

Figure 1 illustrates the convergence in distribution of $\hat{T}_{n}$ to a $\chi^{2}$ distribution. The parameters for this simulation are $d=4$ and $\mathrm{I}=(1,1,1,1)$. Then,  $D^{\mathrm{I}}=6$. For the sample size, we take $n=10000$. The histogram in blue represents the distribution of $\hat{T}_{n}$ and the curve in red is that of the probability distribution function of the $\chi^{2}_{6}$ distribution. We see that the distribution of $\hat{T}_{n}$ is indeed very close to that of the $\chi^{2}_{6}$ one.

\begin{figure}\label{SimulationChi}
\includegraphics[width=8cm]{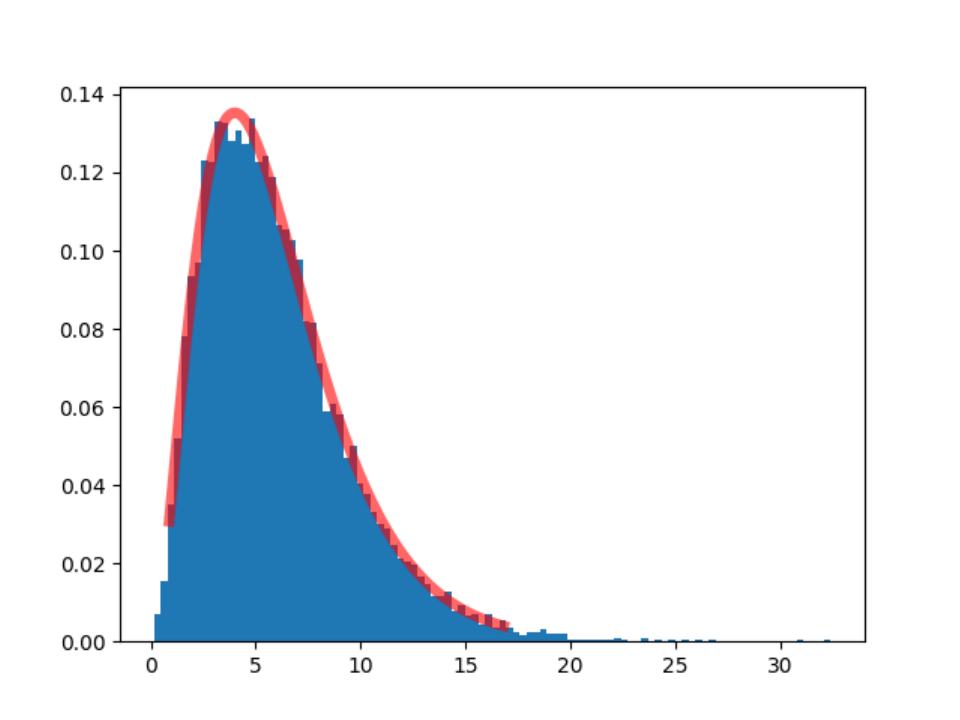}
\setlength\abovecaptionskip{0.01ex}
\caption{Illustration of the convergence of $\hat{T}_{n}$ to the $\chi^{2}_{6}$ distribution}
\end{figure}

\section{Conclusion}\label{sec5}

Given a Gaussian random vector $X$ whose covariance matrix $\Sigma$ has possibly repeated eigenvalues, we develop a geometric method to estimate the flag $F^{\I}(\Sigma)$ of its PS's, for which we provide 
confidence regions
and implementable tests. These results open many questions, among which the following ones, which we present with their motivations. 

\smallskip 

\noindent
$(i)$ Are our geometric CLT's of Theorem $\ref{theoCLTgrass}$ valid when the distribution of $X$ is elliptic? 
\\
If so, then, for such a distribution, the estimation of $F^{\I}(\Sigma)$ would be derived as in paragraphs $\ref{paragraphKhi2}$ and $\ref{paragraphCR}$, which are only based on such CLT's. This question is motivated by \cite{Tyler 1981}. Therein, when $X$ is elliptic, a CLT for each $P_{i}(\ECM )$ is obtained, but not of the form of 
Eq.(\ref{CLTgrass})

\smallskip 

\noindent
$(ii)$ Does $H_{n}^{i}$ converge in $O(q_{i})$ to $E^{i,i}$ wrt the Kullback-Leibler (KL) divergence? \\
By Remark $\ref{lemarkMinGD}$, among all $q_{i}$-frames generating $\rg(\Poi)$, we have that $\cH^{q_{i}}_{\left( \Gamma' P_{i}( \ECM )\Gamma, P_{0}^{i} \right)} \left( E_{n}^{(i)} \right)$ minimizes the geodesic distance in $\St_{q_{i}}$ to $E_{n}^{(i)}$. We conjecture that this geometric optimality should imply a stronger mode of convergence of $H_{n}^{i}$ to $E^{i,i}$, i.e. wrt the KL divergence in the compact group $O(q_{i})$. Our conjecture is motivated by the case of iid convolutions on compact groups, for which the convergence in distribution to the Haar measure has been first obtained: See \cite{Johnson 2004} and references therein. Later, it was proved in \cite{Harremoes 2009} or \cite{Johnson and Suhov 2000}, by information theoretic methods, that they converge wrt the KL divergence.

\section{Appendix: Method for the proof of Theorem $\ref{theoCLTgrass}$}\label{sec6}

This Appendix is devoted to the proof of Theorem $\ref{theoCLTgrass}$ and to some tools developed for it. The order of appearance of the proofs does not follow that of the preceding sections. Instead, the most technical parts are presented at the end. Thus, the proof of the generalized $\delta$-method of Theorem $\ref{AndersonCrit}$ is deferred to the end of this Appendix, i.e. to subsection $\ref{proofAndersonCrit}$. 

Now, we describe the proof of Theorem $\ref{theoCLTgrass}$. By Theorem $\ref{CLTsym}$, $U_{n}$ converges in distribution to $U$. Thus, we prove Theorem $\ref{theoCLTgrass}$ by expressing, in $(\ref{GniHniOfUn})$, the lhs' of $(\ref{CLTgrass})$ and $(\ref{TheoCVtoHaar})$, i.e. $\Gni$ and $\Hni$, as functions of $U_{n}$ and then by deriving, from the generalized $\delta$-method of Theorem $\ref{AndersonCrit}$, that these functions of $U_{n}$ converge to functions of $U$ which are the rhs' of $(\ref{CLTgrass})$ and $(\ref{TheoCVtoHaar})$. However, $\Gni$ and $\Hni$ are defined only when $\pi^{i}\left( E_{n}^{(i)} \right) \notin \Cut(\Poi)$. So, instead of applying Theorem $\ref{AndersonCrit}$ with $U_{n}$, we apply it with  \textit{truncations} of $U_{n}$, in the sense of paragraph $\ref{defTrunc}$.

\subsection{Preliminaries from Anderson's proof}

 First, we present results from the proof of Theorem $\ref{MainAnderson}$ which will be used later for ours. Consider the open subset of $\Sym$ defined by  
\begin{equation}\label{defS0}
\bS_{0}:=\left\{ S \in \Sym : S^{(i,i)} \in \mathrm{Sym}_{q_{i}}^{\neq} , 1 \leq i \leq r \right\}. 
\end{equation}

\noindent
Let $\left(u_{n}\right)_{n \geq 1}$ be a sequence valued in $\left( \SnAstar \right)_{n \geq 1}$ which converges to $u \in \bS_{0}$. Set 
\begin{equation}\label{setEnPn}
\cE_{n}:=e_{n}(u_{n}).
\end{equation}

\begin{proposition}\label{AndersonPseudoCont}
For $1 \leq i,j \leq r$ with $i \neq j$, we have that $\cE_{n}^{(i,i)}$ and $\sqrt{n}\cE_{n}^{(i,j)}$ converge to finite limits denoted respectively by $\cE^{i,i}(u)$ and $\cF^{i,j}(u)$ which satisfy  
\begin{equation}\label{system22GL}
\lambda_{i}\cG^{i,j}(u) + \lambda_{j}\cL^{i,j}(u) = u^{(i,j)} 
\quad \textrm{and} \quad
\cG^{i,j}(u) + \cL^{i,j}(u) = 0, 
\end{equation}

\noindent
with $\cG^{i,j}(u) := \cE^{i,i}(u) \left( \cF^{j,i}(u) \right)'$ and $\cL^{i,j}(u) := \cF^{i,j}(u) \left( \cE^{j,j}(u) \right)'$. Thus, by $(\ref{system22GL})$,
\begin{equation}\label{GijOfU}
\cG^{i,j}(u) = \frac{1}{\lambda_{i}-\lambda_{j}} u^{(i,j)}. 
\end{equation}

\end{proposition}

\noindent
Recall that $U$ is the limit rv of Theorem $\ref{CLTsym}$. Then, the following Lemma is obtained in \cite{Anderson 1963}. 

\begin{lemma}\label{lawLimitFunctions}
$\cE^{i,i}(U)$ is distributed as $E^{i,i}$ in $(i)$ of Theorem $\ref{MainAnderson}$ and $\cF^{i,j}(U)$ is distributed as $F^{i,j}$ in $(ii)$ of Theorem $\ref{MainAnderson}$.
\end{lemma}

\subsection{Expression of $\Gni$ and $\Hni$ as functions of a truncation of $U_{n}$}

\subsubsection{Expression of $\Gni$ and $\Hni$ as functions of $U_{n}$}

First, in $(\ref{EnOfUn})$, we express $E_{n}$ in function of $U_{n}$. Recall that $E_{n}=\psi(T_{n})$, where $\psi : \Symdiff \lra O(d)$ is the eigenvector map of Definition $\ref{DefEigenvectorMap}$. In view  of truncations, we extend the map $\psi$ as follows: see Remark $\ref{remarkValuesAt0}$.

\begin{definition}
Let $\wpsi : \Symdiff \bigcup \left\{ \Delta \right\} \lra O(d)$ be the map such that for $S \in \Symdiff$, $\wpsi(S)=\psi(S)$ and $\wpsi(\Delta) = I_{d}$. Then, $\wpsi$ is called the \textit{extended eigenvector map}. 
\end{definition}

\noindent
$(i)$ By definition, $U_{n}:=\sqrt{n}\left( T_{n}-\Delta \right)$. Then, $T_{n}=\phi_{n}(U_{n})$, where
\begin{equation*}
\phi_{n}(u) = \Delta+ n^{-1/2}u, \quad u \in \Sym. 
\end{equation*}

\noindent
So, $E_{n}=\psi(\phi_{n}(U_{n}))$ provided that $U_{n} \in \phi_{n}^{-1}\left( \Symdiff \right)$. Thus, consider the sets 
\begin{equation*}
\SnAstar := \phi_{n}^{-1}\left( \Symdiff \right)
\quad \textrm{and} \quad
\SnA := \SnAstar \bigcup \left\{0 \right\}. 
\end{equation*}

\noindent
Since $\phi_{n}(0)=\Delta$, we may define the map $e_{n} : \SnA \lra O(d)$ by $e_{n}(u) = \wpsi \left( \phi_{n}(u) \right)$. Then, 
\begin{equation}\label{EnOfUn}
U_{n} \in \SnAstar \implies E_{n} = e_{n}(U_{n}).
\end{equation}

\noindent
$(ii)$ Now, for $1 \leq i \leq r$, we express $\Gni$ and $\Hni$ as functions of $U_{n}$. Define the map $\pni$ by
\begin{equation*}
\pni : \SnA \lra G^{i} \quad \textrm{and} \quad \pni(u) = \pi^{i}\left( e_{n}(u)^{(i)} \right).
\end{equation*}

\noindent
If $\pni(u) \notin \Cut(\Poi)$, then $\gni(u)$ and $\hni(u)$ introduced hereafter are well-defined.
\begin{equation}\label{defgnihni}
\gni(u) := \sqrt{n} \Logi \left( \pni(u) \right) \in \Toi
\enskip \textrm{and} \enskip 
\hni(u) := \cH^{i}_{(\pni(u), \Poi)}\left( e_{n}(u)^{(i)} \right) \in \St^{i}
\end{equation}

\noindent
Thus, we consider the sets  
\begin{equation*}
\Wni := \left\{ u \in \Sym : \pni(u) \notin \Cut(\Poi) \right\}
\quad \textrm{and} \quad
\SniStar := \SnAstar \bigcap \Wni. 
\end{equation*}

\noindent
Now, Equation $(\ref{EnOfUn})$ implies that $\pi^{i}\left( \EnOfi \right) = \pni(U_{n})$ provided that $U_{n} \in \SnAstar$. So, 
\begin{equation}\label{GniHniOfUn}
U_{n} \in \SniStar \implies \Gni = \gni(U_{n}) \enskip \textrm{and} \enskip \Hni = \left( I_{d}^{(i)} \right)' \hni(U_{n}).
\end{equation}

\subsubsection{Truncations of rv's}\label{defTrunc}

Let $M$ be a metric space with Borel $\sigma$-algebra $\cB$. 

\begin{definition}
Let $Z : (\Omega, \cA) \lra (M, \cB)$ be a rv and $A \in \cA$. If $M$ is a vector space, let $Z\mathbf{1}^{+}_{A}$ be the rv such that $(Z\mathbf{1}^{+}_{A})(\omega)=Z(\omega)$ if $\omega \in A$ and $(Z\mathbf{1}^{+}_{A})(\omega)=0 \in M$ else. If $M$ is a subgroup of $GL_{q}(\R)$, denote by $Z\mathbf{1}^{\times}_{A}$ the rv such that, $(Z\mathbf{1}^{\times}_{A})(\omega)=Z(\omega)$ if $\omega \in A$ and $(Z\mathbf{1}^{\times}_{A})(\omega)=I_{q}$ else. Then,  we say that the rv's $Z\mathbf{1}^{+}_{A}$ and $Z\mathbf{1}^{\times}_{A}$ are \textit{truncations} of $Z$ wrt $A$.
\end{definition}

\begin{lemma}\label{truncCV}
Let $(X_{n})_{n \geq 1}$ be a sequence of rv's valued in the metric space $M$ and $(A_{n})_{n \geq 1}$ a sequence of events in $\cA$. Assume that $X_{n}\mathbf{1}^{+}_{A_{n}} \xrightarrow[]{d} X$ or $X_{n}\mathbf{1}^{\times}_{A_{n}} \xrightarrow[]{d} X$ as $n \rightarrow \infty$, where $X$ is a rv valued in $M$. If $P(A_{n}) \xrightarrow[n \rightarrow \infty]{} 1$, then $X_{n} \xrightarrow[n \rightarrow \infty]{d} X$ also.  
\end{lemma}

\begin{proof}
The proof follows readily from the definitions. 
\end{proof}

\subsubsection{Expression of $\Gni$ and $\Hni$ as functions of $\Vni$}

The rv $U_{n}$ is valued in the vector space $\Sym$. So, for $1 \leq i \leq r$, we define below a truncation of $U_{n}$. Thus, set
\begin{equation*}
\Vni:=U_{n} \mathbf{1}^{+}_{\Ani} \quad \textrm{where } \Ani := \left\{ U_{n} \in \SniStar \right\}. 
\end{equation*}

\noindent
So, $\Vni$ is valued in the set $\Sni$ defined hereafter, which holds for all $\omega \in \Omega$ and not only a.s. 
\begin{equation*}
\Sni := \SniStar \bigcup \left\{ 0 \right\}.
\end{equation*}

\begin{remark}\label{remarkValuesAt0}
By definition of the map $\wpsi$, $e_{n}(0) = \wpsi(\phi_{n}(0))=\wpsi(\Delta)=I_{d}$, so that $\pni(0)=\pi^{i}\left( I_{d}^{(i)} \right) = \Poi$. So, $\gni(u)$ and $\hni(u)$ are also well-defined when $u=0$, for which 
\begin{equation}\label{mainValuesAt0}
\gni(0) = 0 \in \Toi
\quad \textrm{and} \quad 
\left( I_{d}^{(i)} \right)' \hni(0) = I_{q_{i}}.
\end{equation}
\end{remark}

\noindent
So, we may consider the maps $\gni : \Sni \lra \Toi$ and $\hni : \Sni \lra \St^{i}$ defined by $(\ref{defgnihni})$. Then, by $(\ref{GniHniOfUn})$ and $(\ref{mainValuesAt0})$, we express hereafter $\Gni \mathbf{1}^{+}_{\Ani}$ and $\Hni \mathbf{1}^{\times}_{\Ani}$ as functions of $\Vni$:
\begin{equation}\label{lhsFunctionTrunc}
\Gni \mathbf{1}^{+}_{\Ani} = \gni(\Vni)
\quad \textrm{and} \quad
\Hni \mathbf{1}^{\times}_{\Ani} = \left( I_{d}^{(i)} \right)' \hni(\Vni). 
\end{equation}

\subsection{Proof of Theorem $\ref{theoCLTgrass}$}

We prove Theorem $\ref{theoCLTgrass}$ by applying Theorem $\ref{AndersonCrit}$. Thus, we define in Table $\ref{AuxiliaryFunctionsGeometric}$ the elements to which we apply Theorem $\ref{AndersonCrit}$. In Table $\ref{AuxiliaryFunctionsGeometric}$, $\bS_{0}$ is the subset defined in $(\ref{defS0})$ and the map $\cE^{i,i} : \bS_{0} \lra O(q_{i})$ is defined in Proposition $\ref{AndersonPseudoCont}$.

\begin{table}[htbp]
\centering
\noindent\begin{tabular}{|c|c|c|}
\hline
 & proof of $(\ref{CLTgrass})$ & proof of $(\ref{TheoCVtoHaar})$ \\
\hline 
Metric spaces $\bS$ and $\T$ & $\bS = \Sym$ and $\T = \Toi$ & $\bS = \Sym$ and $\T = O(q_{i})$ \\
\hline 
Function $g_{n}$ &  $\gni : \Sni \lra \Toi$  &  $\left( I_{d}^{(i)} \right)' \hni : \Sni \lra O(q_{i})$ \\
\hline
Domain $\bS_{n}$ and subset $\bS_{n}^{*}$ &
\multicolumn{2}{|c|}{$\Sni$ ~and~ $\SniStar$} \\
\hline 
Random variable $V_{n}$ & 
\multicolumn{2}{|c|}{$\Vni=U_{n} \bf{1}^{+}_{\left\{U_{n} \in \SniStar \right\}}$} \\ 
\hline
Limit $V$ of $V_{n}$ & 
\multicolumn{2}{|c|}{$\Vni \xrightarrow[n \ra \infty]{d} U$} \\ 
\hline
Limit function $g_{0}$ & $g_{0}^{i} : \bS_{0} \lra \Toi$ & $\cE^{i,i} : \bS_{0} \lra O(q_{i})$ \\
\hline
\end{tabular}
\caption{Elements to which Theorem $\ref{AndersonCrit}$ is applied}
    \label{AuxiliaryFunctionsGeometric}
\end{table}

\subsubsection{Conditions to apply Theorem $\ref{AndersonCrit}$}

We prove hereafter that all the conditions of Theorem $\ref{AndersonCrit}$ hold with the elements of Table $\ref{AuxiliaryFunctionsGeometric}$. Firstly, by $(ii)$ of Proposition $\ref{PropUnSniStar}$ below, for all $1 \leq i \leq r$, $\Vni \xrightarrow[]{d} U$ as $n \ra \infty$. Then, by $(iii)$ of Proposition $\ref{PropUnSniStar}$, Assumption $(\ref{VnSnStar})$ holds. Finally, by Propositions $\ref{pseudoContSym}$ and $\ref{pseudoContHaar}$ below, Assumption $(\ref{pseudoCont})$ holds.

\begin{proposition}\label{PropUnSniStar}
$(i)$ For all $1 \leq i \leq r$, $\Pr \left( \Ani \right) \xrightarrow[]{} 1$ as $n \ra \infty$, where we recall that $\Ani := \left\{ U_{n} \in \SniStar \right\}$. In particular, we deduce that  
\begin{equation*}
\Pr\left( E_{n} \in \W_{\PoI} \right) = \Pr\left( \forall~ 1 \leq i \leq r,~  U_{n} \in \SniStar \right) \xrightarrow[n \ra \infty]{} 0.
\end{equation*}

\noindent
$(ii)$ For all $1 \leq i \leq r$, the sequence $(\Vni)_{n}$ converges in distribution to $U$ as $n \ra \infty$. 

\vspace{.15cm}

\noindent
$(iii)$ For all $1 \leq i \leq r$, $\Pr\left( \Vni \in \SniStar \right) \lra 1$ as $n \ra \infty$. 
\end{proposition}

\begin{proof}
See paragraph $\ref{proofPropUnSniStar}$. 
\end{proof}

\begin{proposition}\label{pseudoContSym}
Let $\left(u_{n}\right)_{n \geq 1}$ be a sequence valued in $\left( \SniStar \right)_{n \geq 1}$ which converges to $u \in \bS_{0}$. Then, $g_{n}^{i}(u_{n}) \lra g_{0}^{i}(u)$ as $n \ra \infty$. 
\end{proposition}

\begin{proof}
See paragraph $\ref{proofPseudoContSym}$.  
\end{proof}

\begin{proposition}\label{pseudoContHaar}
Let $\left(u_{n}\right)_{n \geq 1}$ be as in Proposition $\ref{pseudoContSym}$. Then, $\left( I_{d}^{(i)} \right)' \hni(u_{n}) \lra \cE^{i,i}(u)$ as $n \ra \infty$, where $\cE^{i,i}(u)$ is defined in Proposition $\ref{AndersonPseudoCont}$. 
\end{proposition}

\begin{proof}
See paragraph $\ref{proofPseudoContHaar}$. 
\end{proof}

\subsubsection{End of proof of Theorem $\ref{theoCLTgrass}$}

We deduce that we may apply Theorem $\ref{AndersonCrit}$, from which $(\ref{lhsFunctionTrunc})$ implies that  
\begin{equation}\label{mainTheoEquiv}
\Gni \mathbf{1}^{+}_{\Ani} = \gni(\Vni)  \xrightarrow[n \rightarrow \infty]{d} \goi(U)
\quad \textrm{and} \quad 
\Hni \mathbf{1}^{\times}_{\Ani} = \left( I_{d}^{(i)} \right)'\hni(\Vni) \xrightarrow[n \rightarrow \infty]{d} \cE^{i,i}(U). 
\end{equation}

\noindent
Finally, by Proposition $\ref{PropUnSniStar}$, $\Pr \left( \Ani \right) \lra 1$ as $n \ra \infty$. So, by Lemma $\ref{truncCV}$ and $(\ref{mainTheoEquiv})$, we conclude that 
$(\ref{CLTgrass})$ and $(\ref{TheoCVtoHaar})$ hold, which proves Theorem $\ref{theoCLTgrass}$.

\subsection{Proof of Proposition $\ref{PropUnSniStar}$, $\ref{pseudoContSym}$ and $\ref{pseudoContHaar}$}

\subsubsection{Proof of Proposition $\ref{PropUnSniStar}$}\label{proofPropUnSniStar}

\begin{proof}
$(i)$ By definition, $U_{n} \in \SniStar$ iff $U_{n} \in \SnAstar$ and $\pni(U_{n}) \notin \Cut(\Poi)$. Recall that $U_{n} \in \SnAstar \implies \pni(U_{n}) = \pi^{i} \left( \EnOfi \right)$ and that $U_{n} \in \SnAstar$ a.s. So, since $\Poi = \IdOfi \left( \IdOfi \right)'$, 
\begin{equation*}
\Pr \left( U_{n} \in \SniStar \right) = \Pr \left( \pi^{i} \left( \EnOfi \right) \notin \Cut(\Poi) \right) = \Pr \left( \EnOfi \left( \EnOfi \right)' \notin \Cut \left( \IdOfi \left( \IdOfi \right)' \right) \right). 
\end{equation*}

\noindent
Now, by the description of the cut locus in Grassmannians given in Lemma $\ref{CutGrass}$,
\begin{equation*}
\EnOfi \left( \EnOfi \right)' \notin \Cut(\Poi) \iff \rk\left( \left( \IdOfi \right)' \EnOfi \right) = q_{i} \iff \rk\left( \EnOfii \right) = q_{i}.
\end{equation*}

\noindent
So, $(i)$ holds provided that $\Pr \left( \rk\left( \EnOfii \right)=q_{i} \right) \lra 1$ as $n \ra \infty$, which we derive by applying Lemma $\ref{Lem2.5Tyler}$ below. We check hereafter that its assumptions hold. First, by Theorem \ref{MainAnderson}, $\EnOfii$ converges in distribution to $E^{i,i}$, where $\rk(E^{i,i}) = q_{i}$ a.s. Furthermore, for all $n \geq 1$, $\mathrm{rk}(E_{n}^{(i,i)}) \leq q_{i}$. So, the assumptions of Lemma $\ref{Lem2.5Tyler}$ hold, from which $(i)$ is deduced.  

\noindent\\
$(ii)$ and $(iii)$: By definition of $\SniStar$ and $\Vni$,
\begin{equation}\label{boundUnStar}
\Pr\left( U_{n} \in \SniStar \right) \leq 
\Pr\left( \Vni=U_{n} \textrm{ and } U_{n} \in \SniStar \right). 
\end{equation}

\noindent
On the one hand, the rhs of $(\ref{boundUnStar})$ is bounded by $\Pr\left( \Vni=U_{n} \right)$. So, by $(i)$, $\Pr\left( \Vni=U_{n} \right) \lra 1$ as $n \ra \infty$. Since 
$(U_{n})$ converges in distribution to $U$, we deduce that $(ii)$ holds. On the other hand, the rhs of $(\ref{boundUnStar})$ is bounded by 
$\Pr\left( \Vni \in \SniStar \right)$. So $(i)$ implies that $(iii)$ holds. 
\end{proof}

\noindent
The following Lemma is used in the proof of Proposition $\ref{PropUnSniStar}$. It is proved in \cite{Tyler 1981}.

\begin{lemma}\label{Lem2.5Tyler}
Let $B_{n}$ and $B$ be random matrices of same size. Assume that $B_{n} \xrightarrow[n \rightarrow \infty]{d} B$, that, a.s., $\mathrm{rk}(B) = b$ and that 
$\mathrm{Pr}(\mathrm{rk}(B_{n}) \leq b) \xrightarrow[n \rightarrow \infty]{} 1$. Then, $\mathrm{Pr}(\mathrm{rk}(B_{n}) = b) \xrightarrow[n \rightarrow \infty]{} 1$. 
\end{lemma}

\subsubsection{Preliminaries for the proofs of Propositions $\ref{pseudoContSym}$ and $\ref{pseudoContHaar}$}

Let $\left(u_{n}\right)_{n \geq 1}$ be a sequence valued in $\left( \SniStar \right)_{n \geq 1}$ which converges to $u \in \bS_{0}$, where $\bS_{0}$ is defined in $(\ref{defS0})$. Set 
\begin{equation}\label{setEnPn}
\cE_{n}:=e_{n}(u_{n}) \quad \textrm{and} \quad \Pni:=\pni(u_{n}), \quad n \geq 1, ~1 \leq i \leq r.
\end{equation}

\noindent
Since $\SniStar \subset \SnAstar$, Proposition $\ref{AndersonPseudoCont}$ still holds with this sequence $\left(u_{n}\right)_{n \geq 1}$.

\begin{lemma}\label{EnAnBn}
For $n \geq 1$, there exist matrices $\cA_{n}$ and $\cB_{n}$ such that   
\begin{equation}\label{decompositionEn}
\cE_{n} = \cA_{n} + \rnm \cB_{n},  
\end{equation}

\noindent
where $\cA_{n} \in \cD(\I)$ i.e. $\cA_{n}=\Dg\left( \cA_{n}^{(1,1)}, ..., \cA_{n}^{(i,i)}, ..., \cA_{n}^{(r,r)} \right)$ and the sequences $\left(\cA_{n}\right)_{n \geq 1}$ 
and $\left(\cB_{n}\right)_{n \geq 1}$ converge respectively to finite limits $\cA(u)$ and $\cB(u)$ such that 
\begin{equation}\label{limitsABequalsEF}
\cA(u) \in \cD(\I), \quad
\cA(u)^{(i,i)}=\cE^{i,i}(u) \quad \textrm{and} \quad \cB(u)^{(i,j)}=\cF^{i,j}(u).
\end{equation}
\end{lemma}

\begin{proof}
Set $\cA_{n}:=\Dg\left( e_{n}^{1}(u_{n}), ..., \eni(u_{n}), ..., e_{n}^{r}(u_{n}) \right) \in \cD(\I)$ and $\cB_{n}$ defined by: $\cB_{n}^{(i,j)}=f_{n}^{i,j}(u_{n})$ if $i \neq j$ and $\cB_{n}^{(i,i)}=0$. So, $\eni(u_{n})=\cE_{n}^{(i,i)}$ and for $j \neq i$, we have that $f_{n}^{i,j}(u_{n})=\sqrt{n}\cE_{n}^{(i,j)}$. By Proposition $\ref{AndersonPseudoCont}$, $\cA_{n}$ and $\cB_{n}$ converge to limits satisfying $(\ref{limitsABequalsEF})$.
\end{proof}

\begin{corollary}\label{corollaryEnAnBn}

With the notations of Lemma $(\ref{EnAnBn})$, we have that 
\begin{equation}\label{limitEn(i)}
\cE_{n}^{(i)} \xrightarrow[n \rightarrow \infty]{} 
\begin{pmatrix}
0 & ... & \cE^{i,i}(u) & ... & 0
\end{pmatrix}'
\end{equation}

\noindent
Recall that we have set $\Pni:=\pni(u_{n})$. Then, 
\begin{equation}\label{expressionPni}
\Pni = \left( \cE_{n}^{(i)} \right) \left( \cE_{n}^{(i)} \right)' = \cA_{n}^{(i)}\left( \cA_{n}^{(i)} \right)' + \rnm \Gammani + n^{-1}\Phini, 
\end{equation}

\noindent
where $\Gammani := \cA_{n}^{(i)}\left( \cB_{n}^{(i)} \right)' + \cB_{n}^{(i)}\left( \cA_{n}^{(i)} \right)'$ and $\Phini := \cB_{n}^{(i)}\left( \cB_{n}^{(i)} \right)'$. Furthermore, 
\begin{equation}\label{LimitPni}
\Pni \lra \Poi \quad \textrm{as } n \ra \infty. 
\end{equation}
\end{corollary}

\begin{proof}
$(\ref{limitEn(i)})$ holds by Lemma $\ref{EnAnBn}$ and $(\ref{expressionPni})$ by $(\ref{decompositionEn})$. Finally, by $(\ref{expressionPni})$, since $\cA(u) \in \cD(\I)$, 
\begin{equation}\label{ComputeLimitPni}
\Pni \xrightarrow[n \rightarrow \infty]{} \left( \cA(u)^{(i)} \right) \left( \cA(u)^{(i)} \right)' = \Dg\left( 0_{q_{1}}, ..., \left( \cA(u)^{(i,i)} \right) \left( \cA(u)^{(i,i)} \right)', ..., 0_{q_{r}} \right). 
\end{equation}

\noindent
By $(\ref{limitsABequalsEF})$ and $(\ref{system22GL})$, $\cA(\cU)^{(i,i)}=\cE^{i,i}(u) \in O(q_{i})$, so that the rhs of $(\ref{ComputeLimitPni})$ is equal to $\Poi$.
\end{proof}

\subsubsection{Proof of Proposition $\ref{pseudoContHaar}$}\label{proofPseudoContHaar}

We compute the limit of $\hni(u_{n})$. By definition, 
\begin{equation}\label{}
\hni(u_{n}) = \cH^{i}_{(\Pni, \Poi)}\left( \cE_{n}^{(i)} \right) = \begin{pmatrix}
\cE_{n}^{(i)} & \Deltani \cE_{n}^{(i)}
\end{pmatrix}
 \left( \exp_{m} \begin{pmatrix}
0 & -\Cni \\
I_{q_{i}} & 0
\end{pmatrix}
\right) \begin{pmatrix}
I_{q_{i}} \\ 0_{q_{i}} 
\end{pmatrix}, 
\end{equation}

\noindent
where $\Deltani := \Log^{G^{i}}_{\Pni} \left( \Poi \right)$ and $\Cni := \left( \cE_{n}^{(i)} \right)' \left( \Deltani \right)^{2} \left( \cE_{n}^{(i)} \right)$. By $(\ref{LimitPni})$, 
$\Deltani$ converges to $\Logi \left( \Poi \right) = 0 \in \Toi$, i.e. $\Cni$ converges to $0$. So, setting $h^{i}(u) := 
\begin{pmatrix}
0 & ... & \cE^{i,i}(u) & ... & 0
\end{pmatrix}' \in \St^{i}$, 

\begin{equation}\label{limitHniOfUn}
\hni(u_{n}) \xrightarrow[n \rightarrow \infty]{} \begin{pmatrix}
h^{i}(u) & 0
\end{pmatrix}
 \left( \exp_{m} \begin{pmatrix}
0 & 0 \\
I_{q_{i}} & 0
\end{pmatrix}
\right) \begin{pmatrix}
I_{q_{i}} \\ 0_{q_{i}} 
\end{pmatrix}.
\end{equation}

\noindent
We notice that the rhs of $(\ref{limitHniOfUn})$ is equal to $\cH^{i}_{(\Poi, \Poi)}\left( h^{i}(u) \right) = h^{i}(u)$. Therefore, by $(\ref{limitHniOfUn})$, 
$\left( I_{d}^{(i)} \right)' \hni(u_{n})$ converges to $\left( I_{d}^{(i)} \right)' h^{i}(u) = \cE^{i,i}(u)$ as $n \ra \infty$. This concludes the proof.

\subsubsection{Technical preliminaries for the proof of Proposition $\ref{pseudoContSym}$}
 
\begin{lemma}\label{bracketVanish}
Let $1 \leq i \leq r$. Then, for all $A, B, C \in \cD_{0}^{i}(\I)$ and $M \in \MdR$, 
\begin{equation*}
\left[ A, \Poi \right]=0 \quad \textrm{and} \quad \left[ BMC, \Poi \right]=0. 
\end{equation*}
\end{lemma}

\begin{proof}
This follows from elementary calculations. 
\end{proof}

\begin{lemma}\label{geoSeries}
$\MdR$ is endowed with a norm $\left\| \cdot \right\|$ such that, for $A, B \in \MdR$, \\ 
$\left\| AB \right\| \leq \left\| A \right\|.\left\| B \right\|$. Let $(u_{n})_{n \geq 1}$ be a sequence valued in $\MdR$ and $m \geq 1$.

\noindent
$(i)$ If $u_{n}=o(1)$, then $\sum\limits_{k=m}^{\infty}(u_{n})^{k} = o(1)$ as $n \ra \infty$. 

\noindent
$(ii)$ Assume that $u_{n}=O(a_{n})$, where $(a_{n})_{n \geq 1}$ is a sequence valued in $(0, \infty)$ such that $a_{n} \ra \infty$ as $n \ra \infty$. Then, $\sum\limits_{k=m}^{\infty}(u_{n})^{k} = O((a_{n})^{m})$ as $n \ra \infty$.
\end{lemma}

\begin{proof}
This follows readily from the properties of sums of geometric series. 
\end{proof}

\subsubsection{Proof of Proposition $\ref{pseudoContSym}$}\label{proofPseudoContSym}

\begin{proof}
Set $\Pni := \pni(u_{n})$. Then, $g_{n}^{i}(u_{n}) = \sqrt{n}\Logi \left( \Pni \right)$. By the explicit formula for the Riemannian Logarithm in Grassmannians given by 
$(\ref{explicitLogGrass})$,  
\begin{equation}\label{LogExplicitDirect}
\Logi \left( \Pni \right) = \frac{1}{2} \left[ \log_{m} \left( \left( I_{d} - 2\Pni \right) \cI_{i} \right) , \Poi \right], \quad \textrm{where } \cI_{i} := \left( I_{d} - 2\Poi \right).
\end{equation}

\noindent
By Corollary $\ref{corollaryEnAnBn}$, $\Pni \lra \Poi$ as $n \ra \infty$. So, by $(\ref{LogExplicitDirect})$, $\Logi\left( \Pni \right)$ converges to $0$. Thus, we need to prove that the rate of this convergence is in $O(n^{-1/2})$. Since $\cI_{i}^{2}=I_{d}$, we have that 
\begin{equation}\label{EquAllowExpansion}
\left(I_{d}-2\Pni \right)\cI_{i} = I_{d} + \Kni, \quad \textrm{where } \Kni=o(1).
\end{equation}

\noindent
Then, we expand the $\log_{m}\left( \cdot \right)$ in $(\ref{LogExplicitDirect})$. Thus, by $(\ref{EquAllowExpansion})$, for all $n$ large enough,
\begin{equation}\label{LogiExpanded}
\Logi \left( \Pni \right) = \frac{1}{2} \left[ \log_{m} \left( I_{d} + \Kni \right) , \Poi \right] = \frac{1}{2} \sum\limits_{k=1}^{\infty} \frac{(-1)^{k+1}}{k} \left[ \left(\cK_{n}^{i}\right)^{k} , \Poi \right].
\end{equation}

\noindent
In the rhs of $(\ref{LogiExpanded})$, for $k \geq 1$, the rate of convergence to $0$ of $\left(\Kni \right)^{k}=o(1)$ is not controlled. This explains why we needed a \textit{power series} for the $\log_{m}\left( \cdot \right)$ and not only a Taylor expansion. To prove Proposition $\ref{pseudoContSym}$, we need to compute the limit of $g_{n}^{i}(u_{n})$. By $(\ref{LogiExpanded})$, 
\begin{equation}\label{RecallLogiExpanded}
g_{n}^{i}(u_{n}) = \sqrt{n}\Logi \left( \Pni \right)= \frac{1}{2} \sum\limits_{k=1}^{\infty} \frac{(-1)^{k+1}}{k} \left[ \left(\cK_{n}^{i}\right)^{k} , \Poi \right]
\end{equation}

\noindent
where $\Pni := \pni(u_{n})$ and by $(\ref{expressionPni})$, after calculations, 
\begin{equation}\label{expressionKni}
\Kni := \left(I_{d}-2\Pni \right)\left(I_{d}-2\Poi \right) - I_{d} = \fAni - 2\rnm \Gammani \cI_{i} - 2 n^{-1}\Phini \cI_{i}, 
\end{equation}

\noindent
with $\fAni := - 2\Dg \left( 0_{q_{1}}, ... , I_{q_{i}} - \left(\cA_{n}^{(i,i)}\right) \left( \cA_{n}^{(i,i)} \right)', ..., 0_{q_{r}} \right) = o(1)$. We split the sum in $(\ref{RecallLogiExpanded})$ by isolating the term for $k=1$. Thus, by  
Lemma $\ref{KniVanish}$ below, 
\begin{equation}\label{estimateLogiBis}
\Logi \left( \Pni \right) = \frac{1}{2} \left[ \Kni , \Poi \right] + \sum\limits_{k=2}^{\infty} \frac{(-1)^{k+1}}{k} \left[ \left(\Kni\right)^{k} , \Poi \right] = \frac{1}{2} \left[ \Kni , \Poi \right] + o(\rnm).  
\end{equation}

\noindent
Now, we deal with the first term in the rhs of $(\ref{estimateLogiBis})$. By Lemma $\ref{FinalCV}$ below, 
\begin{equation}\label{estimateKPi}
\frac{\sqrt{n}}{2} \left[ \Kni , \Poi \right] = g_{0}^{i}(u) + O(\rnm) = g_{0}^{i}(u) + o(1).
\end{equation}

\noindent
Finally, by $(\ref{estimateLogiBis})$ and $(\ref{estimateKPi})$, $\gni(u_{n}) = \sqrt{n} \Logi \left( \Pni \right) = g_{0}^{i}(u) + o(1)$.   
\end{proof}

\subsubsection{Auxiliary results for the proof of Proposition $\ref{pseudoContSym}$}

\begin{lemma}\label{KniVanish}
Recall that the sequence $(\cK_{n}^{i})_{n}$ is defined in $(\ref{expressionKni})$. Then, 
\begin{equation}\label{LemmaSumFrom2}
\sum\limits_{k=2}^{\infty} \frac{(-1)^{k+1}}{k} \left[ \left(\Kni \right)^{k} ,  \Poi \right] = o(\rnm).
\end{equation}

\end{lemma}

\begin{proof}
\textit{First part:} \enskip 
By the expression of $\Kni$ in $(\ref{expressionKni})$,
\begin{equation}\label{estimateKni}
\Kni = \fAni + \Mni, \quad \textrm{where } \fAni \in \DoiI, \quad \textrm{with }\fAni=o(1) \textrm{ and } \Mni=O(\rnm).
\end{equation}

\noindent
In the rhs of $(\ref{RecallLogiExpanded})$, $\left(\Kni \right)^{k}$ still contains some $\left(\fAni \right)^{\ell}=o(1)$, whose rate of convergence to $0$ is not controlled. Lemma $\ref{bracketVanish}$ provides simplifications. Thus, by $(\ref{estimateKni})$, for all $k \geq 2$,
\begin{equation}\label{KniPowers}
\left(\Kni \right)^{k} = \left(\fAni \right)^{k} + \left( \sum\limits_{\ell=0}^{k-1} \left(\fAni \right)^{\ell} \left(\Mni \right)^{k-\ell} \right) + 
\left( \sum\limits_{\ell=1}^{k} \left(\cM_{n}^{i}\right)^{\ell} \left(\fAni \right)^{k-\ell} \right) + \Rni,
\end{equation}

\noindent
where $\Rni$ is a sum of terms of the form $\left(\fAni \right)^{\ell}\left(\Mni \right)^{\ell'}\left(\fAni \right)^{\ell''}$, with $\ell+\ell'+\ell''=k$ and $\ell>1$, $\ell''>1$. Since $\fAni \in \DoiI$, Lemma $\ref{bracketVanish}$ implies that 
\begin{equation}\label{vanishBracket}
\left[ \left(\fAni \right)^{k} , \Poi \right] = 0
\quad \textrm{and} \quad
\left[ \Rni, \Poi \right]=0.
\end{equation}

\noindent
By combining $(\ref{KniPowers})$ and $(\ref{vanishBracket})$, 
\begin{equation}\label{LogiSimpleBracket}
\sum\limits_{k=2}^{\infty} \frac{(-1)^{k+1}}{k} \left[ \left(\cK_{n}^{i}\right)^{k} ,  P_{0}^{i} \right] =
\frac{1}{2} \left[ \alpha_{n}^{i} ,  P_{0}^{i} \right] + \frac{1}{2} \left[ \beta_{n}^{i} ,  P_{0}^{i} \right], 
\end{equation}

\noindent
where, setting $\epsilon_{k}:=\frac{(-1)^{k+1}}{k}$, 
\begin{equation*}
\alpha_{n}^{i} = \sum\limits_{k=2}^{\infty} \sum\limits_{\ell=0}^{k-1} \epsilon_{k} \left(\fA_{n}^{i}\right)^{\ell} \left(\cM_{n}^{i}\right)^{k-\ell}
\quad \textrm{and} \quad
\beta_{n}^{i} = \sum\limits_{k=2}^{\infty} \sum\limits_{\ell=1}^{k} \epsilon_{k} \left(\cM_{n}^{i}\right)^{\ell} \left(\fA_{n}^{i}\right)^{k-\ell}.
\end{equation*}

\noindent
\textit{Second part:} \enskip
We estimate $\alpha_{n}^{i}$ and $\beta_{n}^{i}$ as $n \ra \infty$. In the double sum defining $\alpha_{n}^{i}$, we wish to factorize $\left(\fAni \right)^{\ell}$ in sums indexed by $k$, by inverting the summation indices. Namely, we invert on the following triangular domain $\tau := \left\{ (k,\ell) \in \N^{2} : k \geq 2 , 1 \leq \ell \leq k-1 \right\}$. Thus, for fixed $k$, we split the sum in $\ell$ into two parts: the terms corresponding to $\ell=0$ and those to $\ell \geq 1$:  
\begin{equation}\label{decAlphAni}
\alpha_{n}^{i} = \left( \sum\limits_{k=2}^{\infty} \epsilon_{k} \left(\cM_{n}^{i}\right)^{k} \right) + 
\left( \sum\limits_{k=2}^{\infty} \sum\limits_{\ell=1}^{k-1} \epsilon_{k} \left(\fAni \right)^{\ell} \left(\cM_{n}^{i}\right)^{k-\ell} \right).
\end{equation}

\noindent
Then, in the second part of the rhs in $(\ref{decAlphAni})$, we invert the indices, which lie in $\tau$.
\begin{equation}\label{doubleSum}
\sum\limits_{k=2}^{\infty} \sum\limits_{\ell=1}^{k-1} \epsilon_{k} \left(\fA_{n}^{i}\right)^{\ell} \left(\cM_{n}^{i}\right)^{k-\ell} = 
\sum\limits_{\ell=1}^{\infty} \sum\limits_{k=\ell+1}^{\infty} \epsilon_{k} \left(\fA_{n}^{i}\right)^{\ell} \left(\cM_{n}^{i}\right)^{k-\ell} = 
\sum\limits_{\ell=1}^{\infty} \left(\fA_{n}^{i}\right)^{\ell} \left( \sum\limits_{k'=1}^{\infty} \epsilon_{\ell+k'} \left(\cM_{n}^{i}\right)^{k'} \right). 
\end{equation}

\noindent
Since $\Mni=O(\rnm)$ and $\fAni=o(1)$, Lemma $\ref{geoSeries}$ implies that 
\begin{equation}\label{estimateMni}
\sum\limits_{k'=1}^{\infty} \epsilon_{\ell+k'} \left(\Mni \right)^{k'} = O(\rnm), \enskip \textrm{uniformly in } \ell, 
\quad \textrm{and} \quad \sum\limits_{\ell=1}^{\infty} \left(\fAni \right)^{\ell} = o(1). 
\end{equation}

\noindent
Therefore, by combining $(\ref{doubleSum})$ and $(\ref{estimateMni})$, 
\begin{equation}\label{2dTermAlphAni}
\sum\limits_{k=2}^{\infty} \sum\limits_{\ell=1}^{k-1} \epsilon_{k} \left(\fAni \right)^{\ell} \left(\Mni \right)^{k-\ell} = o(\rnm).
\end{equation}

\noindent
By Lemma $\ref{geoSeries}$, since $\Mni=O(\rnm)$, we obtain for the first part of the rhs in $(\ref{decAlphAni})$:  
\begin{equation}\label{1stTermAlphAni}
\sum\limits_{k=2}^{\infty} \epsilon_{k} \left(\Mni \right)^{k} = O\left( \left(\rnm \right)^{2}\right) = O\left(n^{-1}\right). 
\end{equation}

\noindent
By $(\ref{decAlphAni})$, $(\ref{2dTermAlphAni})$ and $(\ref{1stTermAlphAni})$, we deduce that $\alpha_{n}^{i}=o(\rnm)$. Similarly, $\beta_{n}^{i}=o(\rnm)$.
\end{proof}

\begin{lemma}\label{FinalCV}
We have that 
\begin{equation}
\frac{\sqrt{n}}{2} \left[ \cK_{n}^{i} , P_{0}^{i} \right] = g_{0}^{i}(u) + O(\rnm).
\end{equation}

\end{lemma}

\begin{proof}
Since $\left[ \fAni , \Poi \right]=0$, the expression of $\cK_{n}^{i}$ in $(\ref{expressionKni})$ yields that  
\begin{equation}
\frac{1}{2} \left[ \Kni ,  \Poi \right] = \frac{1}{2} \left[ \fAni - 2\rnm \Gammani \cI_{i} - 2 n^{-1}\Phini \cI_{i} , \Poi \right]
= - \rnm \left[ \Gammani \cI_{i} , \Poi \right] + O(n^{-1}). 
\end{equation}

\noindent
Thus, it is enough to prove that $-\left[ \Gammani \cI_{i} , P_{0}^{i} \right] \lra g_{0}^{i}(u)$ as $n \ra \infty$. Consider the matrices 
\begin{equation*}
N^{i}_{<} := 
\begin{pmatrix}
\cA_{n}^{(i,i)}\left( \cB_{n}^{(1,i)} \right)' & ... & \cA_{n}^{(i,i)}\left( \cB_{n}^{(i-1,i)} \right)'
\end{pmatrix}
\quad \textrm{and}\quad
N^{i}_{>} := 
\begin{pmatrix}
\cA_{n}^{(i,i)}\left( \cB_{n}^{(i+1,i)} \right)' & ... & \cA_{n}^{(i,i)}\left( \cB_{n}^{(r,i)} \right)'
\end{pmatrix}
\end{equation*}

\noindent
of respective sizes $q_{i} \times (q_{1}+...+q_{i-1})$ and $q_{i} \times (q_{i+1}+...+q_{r})$. After calculations:
\begin{equation*}
-\left[ \Gamma_{n}^{1}\cI_{1} , P_{0}^{1} \right] = 
\begin{pmatrix}
0_{q_{1}} & N^{1}_{>} \\
\left(N^{1}_{>}\right)' & 0 \\
\end{pmatrix}
\qquad , \qquad 
-\left[ \Gamma_{n}^{r}\cI_{r} , P_{0}^{r} \right] = 
\begin{pmatrix}
0 & \left(N^{r}_{<} \right)' \\
N^{r}_{<} & 0_{q_{r}} \\
\end{pmatrix}
\end{equation*}

\begin{equation*}
-\left[ \Gammani \cI_{i} , \Poi \right] =
\begin{pmatrix}
0 & \left(N^{i}_{<}\right)' & 0 \\
N^{i}_{<} & 0_{q_{i}} & N^{i}_{>} \\
0 & \left(N^{i}_{>}\right)' & 0
\end{pmatrix}, 
\quad i \neq 1, r. 
\end{equation*}

\noindent
By $(\ref{limitsABequalsEF})$ in Lemma $\ref{EnAnBn}$, for $j \neq i$, 
\begin{equation}\label{}
\cA_{n}^{(i,i)}\left( \cB_{n}^{(j,i)} \right)' \xrightarrow[n \rightarrow \infty]{} \cE^{i,i}(u) \left( \cF^{j,i}(u) \right)' =: \cG^{i,j}(u)
= \frac{1}{\lambda_{i}-\lambda_{j}}u^{(i,j)}.
\end{equation}

\noindent
By definition of the functions $g_{0}^{i}$, this concludes the proof. 
\end{proof}

\subsection{Proof of Theorem $\ref{AndersonCrit}$}\label{proofAndersonCrit}

\begin{proof}
It is enough to prove that for any closed set $T \subset \T$, 
\begin{equation}\label{goalCVinD}
\varlimsup \Pr ( V_{n} \in g_{n}^{-1}(T) ) \leq \Pr(V \in g_{0}^{-1}(T) ). 
\end{equation}
 
\noindent
Indeed, if $(\ref{goalCVinD})$ holds, then the Portmanteau Theorem allows to conclude the proof. First, 
\begin{equation}\label{sumOfLimsup}
\varlimsup \Pr( V_{n} \in g_{n}^{-1}(T) ) \leq \varlimsup \Pr( V_{n} \in g_{n}^{-1}(T) \cap \bS_{n}^{*} ) + \varlimsup \Pr\left( V_{n} \in g_{n}^{-1}(T) \cap \left(\bS_{n}^{*}\right)^{c} \right).
\end{equation}

\noindent
By $(\ref{VnSnStar})$, the second term of the rhs in $(\ref{sumOfLimsup})$ vanishes. So, 
\begin{equation}\label{addSnStarOK}
\varlimsup \Pr ( V_{n} \in g_{n}^{-1}(T) ) \leq \varlimsup \Pr( V_{n} \in g_{n}^{-1}(T) \cap \bS_{n}^{*} ). 
\end{equation}

\noindent
Consider the following sequence of sets $(R_{m})_{m \geq 1}$ and the set $R^{*}$ defined by 
\begin{equation*}
R_{m} := \bigcup_{n \geq m} \left\{ g_{n}^{-1}(T) \cap \bS_{n}^{*} \right\}, \enskip m \geq 1
\quad \textrm{and} \quad 
R^{*} := \bigcap_{m \geq 1} \overline{R_{m}}. 
\end{equation*}

\noindent
Fix $m \geq 1$. Then, for all $n \geq m$, $\left\{g_{n}^{-1}(T) \cap \bS_{n}^{*} \right\} \subset \overline{R_{m}}$. Since $V_{n} \xrightarrow{d} V$, 
\begin{equation}\label{applyCVinD}
\varlimsup \Pr( V_{n} \in g_{n}^{-1}(T) \cap \bS_{n}^{*} ) \leq \varlimsup \Pr( V_{n} \in \overline{R_{m}} ) \leq \Pr( V \in \overline{R_{m}} ). 
\end{equation}

\noindent
Since the lhs of $(\ref{applyCVinD})$ is independent of $m$, taking the limit as $m \ra \infty$ in $(\ref{applyCVinD})$ yields that 
\begin{equation}\label{limitInM}
\varlimsup \Pr( V_{n} \in g_{n}^{-1}(T) \cap \bS_{n}^{*} ) \leq \Pr( V \in R^{*} ). 
\end{equation}

\noindent
Indeed, the sequence $(R_{m})_{m \geq 1}$ is \textit{decreasing}. Now, we claim that 
\begin{equation}\label{claimRstar}
R^{*} \subset g_{0}^{-1}(T) \cup (\bS_{0})^{c}.
\end{equation}

\noindent
Indeed, let $v \in R^{*}$. Then, by definition of $R^{*}$, we can construct a sequence $(v_{m})_{m \geq 1}$ valued in $(R_{m})_{m \geq 1}$ which converges to $v$. Now, by definition of $R_{m}$, we may build a subsequence $(v_{\phi(m)})_{m \geq 1}$ of $(v_{m})_{m}$ such that for all $m \geq 1$, $v_{\phi(m)} \in g_{\phi(m)}^{-1}(T) \cap \bS_{\phi(m)}^{*}$. Thus, $g_{\phi(m)}(v_{\phi(m)}) \in T$ and, since $v_{\phi(m)} \in \bS_{\phi(m)}^{*}$, assumption $(\ref{pseudoCont})$ implies that $g_{\phi(m)}(v_{\phi(m)}) \lra g_{0}(v)$ as $m \ra \infty$. Since $T$ is closed, we deduce that $v \in g_{0}^{-1}(T)$. This proves the claim $(\ref{claimRstar})$. Thus, by $(\ref{claimRstar})$ and the assumption $\Pr\left( V \in \bS_{0} \right)=1$,  
\begin{equation}\label{ineqClaimStar}
\Pr( V \in R^{*} ) \leq \Pr\left( V \in g_{0}^{-1}(T) \cup (\bS_{0})^{c} \right) \leq \Pr\left( V \in g_{0}^{-1}(T) \right).
\end{equation}

\noindent
We combine $(\ref{addSnStarOK})$, $(\ref{limitInM})$ and $(\ref{ineqClaimStar})$ to conclude that $(\ref{goalCVinD})$ holds. 
\end{proof}

\subsection*{Acknowledgements}
The authors have received funding from the European Research Council (ERC) under the European Union's Horizon 2020 research and innovation program (grant agreement G-Statistics No 786854). It was also supported by the French government through the 3IA C\^ote dAzur Investments ANR-19-P3IA-0002 managed by the National Research Agency.

\bibliographystyle{plain}

\end{document}